\documentclass[11pt,a4paper,reqno]{amsart}
\usepackage[english]{babel}
\usepackage[T1]{fontenc}
\usepackage{verbatim}
\usepackage{palatino}
\usepackage{amsmath}
\usepackage{mathabx}
\usepackage{amssymb}
\usepackage{amsthm}
\usepackage{amsfonts}
\usepackage{graphicx}
\usepackage{esint}
\usepackage{color}
\usepackage{mathtools}
\usepackage{overpic}

\usepackage[colorlinks = true, citecolor = black]{hyperref}
\author{Tuomas Orponen}
\title[$ABC$ sum-product problem]{Hausdorff dimension bounds for the \\ $ABC$ sum-product problem}
\address{Department of Mathematics and Statistics\\ University of Jyv\"askyl\"a,
P.O. Box 35 (MaD)\\
FI-40014 University of Jyv\"askyl\"a\\
Finland}
\email{tuomas.t.orponen@jyu.fi}
\date{\today}
\thanks{T.O. is supported by the Academy of Finland via the projects \emph{Quantitative rectifiability in Euclidean and non-Euclidean spaces} and \emph{Incidences on Fractals}, grant Nos. 309365, 314172, 321896.}

\newcommand{\R}{\mathbb{R}}

\newcommand{\N}{\mathbb{N}}

\newcommand{\Z}{\mathbb{Z}}

\newcommand{\spt}{\operatorname{spt}}
\newcommand{\Hd}{\dim_{\mathrm{H}}}

\newcommand{\dist}{\operatorname{dist}}

\def\Barint_#1{\mathchoice
          {\mathop{\vrule width 6pt height 3 pt depth -2.5pt
                  \kern -8pt \intop}\nolimits_{#1}}%
          {\mathop{\vrule width 5pt height 3 pt depth -2.6pt
                  \kern -6pt \intop}\nolimits_{#1}}%
          {\mathop{\vrule width 5pt height 3 pt depth -2.6pt
                  \kern -6pt \intop}\nolimits_{#1}}%
          {\mathop{\vrule width 5pt height 3 pt depth -2.6pt
                  \kern -6pt \intop}\nolimits_{#1}}}

\numberwithin{equation}{section}

\theoremstyle{plain}
\newtheorem{thm}[equation]{Theorem}
\newtheorem*{"thm"}{"Theorem"}

\newtheorem{lemma}[equation]{Lemma}

\newtheorem{cor}[equation]{Corollary}

\newtheorem*{counter}{Counter assumption}

\theoremstyle{definition}

\newtheorem{definition}[equation]{Definition}

\theoremstyle{remark}
\newtheorem{remark}[equation]{Remark}

\newcommand{\nref}[1]{(\hyperref[#1]{#1})}

\DeclareMathSymbol{\intop}  {\mathop}{mathx}{"B3}

\begin{document} 

\begin{abstract} The purpose of this paper is to complete the proof of the following result. Let $0 < \beta \leq \alpha < 1$ and $\kappa > 0$. Then, there exists $\eta > 0$ such that whenever $A,B \subset \R$ are Borel sets with $\Hd A = \alpha$ and $\Hd B = \beta$, then
\begin{displaymath} \Hd \{c \in \R : \Hd (A + cB) \leq \alpha + \eta\} \leq \tfrac{\alpha - \beta}{1 - \beta} + \kappa. \end{displaymath}
This extends a result of Bourgain from 2010, which contained the case $\alpha = \beta$.

This paper is a sequel to the author's previous work from 2021 which, roughly speaking, established the same result with $\Hd (A + cB)$ replaced by $\dim_{\mathrm{B}}(A + cB)$, the box dimension of $A + cB$. It turns out that, at the level of $\delta$-discretised statements, the superficially weaker box dimension result formally implies the Hausdorff dimension result. \end{abstract}

\maketitle

\tableofcontents

\section{Introduction} 

This paper is a follow-up to \cite{2021arXiv211002779O}, whose main result had the following form. Let $\delta \in (0,1)$, and assume that $A,B,C$ are $\delta$-separated subsets of $[0,1]$. Under \emph{suitable assumptions} on $A,B,C$, and if $\delta,\epsilon > 0$ are small enough, there exists $c \in C$ such that $|A + cB|_{\delta} \geq \delta^{-\epsilon}|A|$. Here $|\cdot|_{\delta}$ is the $\delta$-covering number. The \emph{suitable assumptions} are those in Theorem \ref{main} (with $C := \spt(\nu)$), but their precise form is not relevant quite yet. One could easily deduce the following "continuous" corollary from the result above: under suitable assumptions on $A,B,C \subset [0,1]$, there exists $c \in C$ such that 
\begin{displaymath} \dim_{\mathrm{B}} (A + cB) \geq \Hd A + \epsilon. \end{displaymath}
Here $\dim_{\mathrm{B}}$ and $\Hd$ stand for box and Hausdorff dimensions, respectively. It would be more satisfactory to know that $\Hd (A + cB) \geq \Hd A + \epsilon$, and achieving this upgrade is the main purpose of the present paper.

What kind of $\delta$-discretised statement is needed to prove $\Hd (A + cB) \geq \Hd A + \epsilon$? This is well-known, see e.g. \cite[Theorem 1]{MR4148151}. One needs to know that if $A,B,C$ are $\delta$-separated subsets of $[0,1]$, then under suitable assumptions on $A,B,C$ one can locate $c \in C$ with the following property: for every subset $G \subset A \times B$ with $|G| \geq \delta^{\epsilon}|A||B|$, one has $|\pi_{c}(G)|_{\delta} \geq \delta^{-\epsilon}|A|$, where $\pi_{c}(x,y) = x + cy$. This follows from \cite{2021arXiv211002779O} in the special case $G = A \times B$, but this is too weak to yield any information about $\Hd (A + cB)$. Thus, the main result of \cite{2021arXiv211002779O} needs to be upgraded as follows:

\begin{thm}\label{main} Let $0 < \beta \leq \alpha < 1$ and $\kappa > 0$. Then, for every $\gamma \in ((\alpha - \beta)/(1 - \beta),1]$, there exist $\epsilon_{0},\epsilon,\delta_{0} \in (0,\tfrac{1}{2}]$, depending only on $\alpha,\beta,\gamma,\kappa$, such that the following holds. Let $\delta \in 2^{-\N}$ with $\delta \in (0,\delta_{0}]$, and let $A,B \subset (\delta \cdot \Z) \cap [0,1]$ satisfy the following hypotheses:
\begin{enumerate}
\item[(A)] \label{A} $|A| \leq \delta^{-\alpha}$.
\item[(B)] \label{B} $|B| \geq \delta^{-\beta}$, and $B$ satisfies the following Frostman condition: 
\begin{displaymath} |B \cap B(x,r)| \leq r^{\kappa}|B|, \qquad \delta \leq r \leq \delta^{\epsilon_{0}}. \end{displaymath} 
\end{enumerate}
Further, let $\nu$ be a Borel probability measure with $\spt (\nu) \subset [\tfrac{1}{2},1]$, and satisfying the Frostman condition $\nu(B(x,r)) \leq r^{\gamma}$ for $x \in \R$ and $0 < r \leq \delta^{\epsilon_{0}}$. Then, there exists a point $c \in \spt(\nu)$ such that the following holds: if $G \subset A \times B$ is any subset with $|G| \geq \delta^{\epsilon}|A||B|$, then 
\begin{displaymath} |\pi_{c}(G)| \geq \delta^{-\epsilon}|A|. \end{displaymath}
Here $\pi_{c}(x,y) := x + cy$.
\end{thm}

With Theorem \ref{main} in hand, one can deduce the following corollary about $\Hd (A + cB)$ with standard arguments, which are nevertheless recorded in Section \ref{appA}:
\begin{cor}\label{hausdorffCor} Let $0 < \beta \leq \alpha < 1$ and $\kappa > 0$. Then, there exists $\eta = \eta(\alpha,\beta,\kappa) > 0$ such that if $A,B \subset \R$ are Borel sets with $\Hd A = \alpha$, $\Hd B = \beta$, then 
\begin{displaymath} \Hd \{c \in \R : \Hd (A + cB) \leq \alpha + \eta\} \leq \tfrac{\alpha - \beta}{1 - \beta} + \kappa. \end{displaymath}
In particular, $\Hd \{c \in \R : \Hd (A + cB) = \alpha\} \leq (\alpha - \beta)/(1 - \beta)$.
\end{cor}
The case $\alpha = \beta$ of Corollary \ref{hausdorffCor} is due to Bourgain \cite{Bourgain10}. In comparison, a classical exceptional set estimate of Kaufman \cite{Ka} would yield the weaker estimate $\Hd \{c \in \R : \Hd (A + cB) = \alpha\} \leq \alpha$, which is independent of $\beta = \Hd B$. Another classical estimate, due to Falconer \cite{MR673510} and Peres-Schlag \cite{MR1749437} shows that $\Hd \{c \in \R : \Hd (A + cB) = \alpha\} \leq 1 - \beta$. This is weaker than Corollary \ref{hausdorffCor} for $0 < \beta \leq \alpha < \tfrac{3}{4}$, but if $\alpha > \tfrac{3}{4}$, is stronger for e.g. $\beta = \tfrac{1}{2}$. It seems plausible that the optimal bound is $\Hd \{c \in \R : \Hd (A + cB) \leq \alpha\} \leq \alpha - \beta$. For a more extensive discussion on this conjecture, and related literature, see the introduction to \cite{2021arXiv211002779O}.

\subsection{Proof outline} The main theorem of \cite{2021arXiv211002779O} is recorded in Theorem \ref{mainTechnical} below: it is simply Theorem \ref{main} with the weaker conclusion $|A + cB|_{\delta} \geq \delta^{-\epsilon}|A|$. In summary, the outline of the paper is: \emph{Theorem \ref{mainTechnical} formally implies Theorem \ref{main}.} I mention that the threshold $\gamma > (\alpha - \beta)/(1 - \beta)$ of Theorem \ref{main} plays no role in the present paper, except that it is required by Theorem \ref{mainTechnical}. If a variant of Theorem \ref{mainTechnical} were known with, say, $\gamma > \alpha - \beta$, then the argument of this paper would demonstrate Theorem \ref{main} (and Corollary \ref{hausdorffCor}) with the same threshold. The lower bound $\gamma > \alpha - \beta$ would be sharp if true, see \cite{2021arXiv211002779O}. 

The reduction from Theorem \ref{main} to Theorem \ref{mainTechnical} proceeds in several stages. First, in Section \ref{s:toy}, we prove the following toy version of Theorem \ref{main}: instead of allowing for general subsets of the form $G \subset A \times B$ with $|G| \geq \delta^{\epsilon}|A||B|$, this version (Theorem \ref{mainSubset1}) only treats subsets of the form $G = A \times B'$ with $|B'| \geq \delta^{\epsilon}|B|$. The conclusion is that there exists $c \in \spt(\nu)$ such that $|A + cB'| \geq \delta^{-\epsilon}|A|$ for all $B' \subset B$ with $|B'| \geq \delta^{\epsilon}|B|$. 

Even the toy version, Theorem \ref{mainSubset1}, is not proved directly: we will pass through a toy-toy version, Theorem \ref{mainSubset2}, where we are first allowed to replace $A \times B$ by a subset of the form $A \times \bar{B}$, and then the conclusion explained above is established for $A \times \bar{B}$ in place of $A \times B$. Fortunately, the passage between the toy and toy-toy versions can be accomplished by a formal exhaustion argument, which I learned from He's paper \cite{MR4148151}.

The toy-toy version is eventually deduced, in Section \ref{s:subsetReduction}, by a direct argument from the main result in \cite{2021arXiv211002779O}. This is the core of the paper. Instead of giving details here, I mention a key difficulty: this reduction, and various other steps of the argument (including those in \cite{2021arXiv211002779O}) would be simpler if we \emph{a priori} knew that 
\begin{equation}\label{form24} |A + A| \approx |A| \quad \text{and} \quad |B + B| \approx |B|. \end{equation}
(In this heuristic discussion, I will leave the meaning of "$\approx$" to the reader's imagination.) In the case $|A| \approx |B|$, treated by Bourgain in \cite{Bourgain10}, this is automatic: if $|A + cB|_{\delta} \approx |A| \approx |B|$ for some $c \in [\tfrac{1}{2},1]$, then \eqref{form24} holds by Pl\"unnecke's inequality. However, in our situation $B$ is typically much smaller than $A$, and now the property $|A + cB|_{\delta} \approx |A|$ implies neither property in \eqref{form24}. Nevertheless, \eqref{form24} is needed, technically because Lemma \ref{OVLemma} is useless without \eqref{form24}. Roughly speaking, Theorem \ref{mainSubset2} is proved by making a counter assumption, and using it to generate new sets $\bar{A} \neq A$ and $\bar{B} \neq B$ which satisfy the original hypotheses, and additionally \eqref{form24}. At some level, this argument is reminiscent of the proof of the asymmetric Balog-Szemer\'edi-Gowers theorem in \cite{MR2289012} (see Theorem \ref{BSG}).

Once we have the toy version, Theorem \ref{mainSubset1}, at our disposal, it remains to deduce Theorem \ref{main} from Theorem \ref{mainSubset1}. This step is \emph{based} on the asymmetric Balog-Szemer\'edi-Gowers theorem -- unlike the other steps. We make a counter assumption that for every $c \in \spt(\nu)$ there exists a subset $G_{c} \subset A \times B$ with $|G| \gtrapprox |A||B|$ such that $|\pi_{c}(G)|_{\delta} \lessapprox |A|$. By the B-S-G theorem, this yields for every $c \in \spt(\nu)$ subsets $A_{c} \subset A$ and $B_{c} \subset B$ such that $|A_{c}| \gtrapprox |A|$, $|B_{c}| \gtrapprox |B|$, and $|A_{c} + cB_{c}| \lessapprox |A|$. With the help of probabilistic arguments, and the Pl\"unnecke-Ruzsa inequality (Lemma \ref{PRIneq}), this allows us to construct a new $\delta$-separated set $H \subset [0,1]$ with $|H| \lessapprox |A|$, and a subset $C \subset \spt(\nu)$ with $\nu(C) \gtrapprox 1$, such that $|H + cB_{c}| \lessapprox |H|$ for all $c \in C$. This violates the first toy version, Theorem \ref{mainSubset1}, applied to $H,B$ and finally concludes the proof of Theorem \ref{main}.

\subsection{Notation}  The notation $|A|$ stands for the cardinality of a finite set $A \subset \R^{d}$, typically $A \subset \delta \cdot \Z^{d}$ for some $\delta > 0$. For a dyadic rational $r \in 2^{-\N}$, and a bounded set $A \subset \R$, we write $|A|_{r}$ for the least number of dyadic intervals of length $\delta$ required to cover $A$ (in the introduction, we used the same notation for the $r$-covering number, which is comparable up to a multiplicative constant). We will also write $A(r)$ for the open $r$-neighbourhood of $A$, and $A_{r} = (r \cdot \Z) \cap A(r)$ (this notation is only used with $r = \delta$).

If $f,g \geq 0$, the notation $f \lesssim g$ means that there exists an absolute constant $C > 0$ such that $f \leq Cg$. If the constant $C$ is allowed to depend on a parameter "$p$", this is signified by writing $f \lesssim_{p} g$. The two-sided inequality $f \lesssim g \lesssim f$ is abbreviated $f \sim g$. The notation $f \lessapprox g$ and $f \approx g$ is only used properly in Section \ref{appA}, and is explained there.

\section{A toy version of the main theorem}\label{s:toy}

Theorem \ref{main} claims the existence of $c \in \spt(\nu)$ such that $|\pi_{c}(G)| \geq \delta^{-\epsilon}|A|$ for all $G \subset A \times B$ with $|G| \geq \delta^{\epsilon}|A||B|$. A toy version of this problem is: find $c \in \spt(\nu)$ such that $|A + cB'| \geq \delta^{-\epsilon}|A|$ for all $B' \subset B$ with $|B'| \geq \delta^{\epsilon}|B|$. Instead of approaching Theorem \ref{main} directly, we will first prove this toy version, Theorem \ref{mainSubset1}, in the present section. After this has been accomplished, Theorem \ref{main} is proved in full generality in Section \ref{s:mainProof}.

\begin{thm}\label{mainSubset1} Let $0 < \beta \leq \alpha < 1$ and $\kappa > 0$. Then, for every $\gamma \in ((\alpha - \beta)/(1 - \beta),1]$, there exist $\epsilon_{0},\epsilon,\delta_{0} \in (0,\tfrac{1}{2}]$, depending only on $\alpha,\beta,\gamma,\kappa$, such that the following holds. Let $\delta \in 2^{-\N}$ with $\delta \in (0,\delta_{0}]$, and let $A,B \subset (\delta \cdot \Z) \cap [0,1]$ satisfy the following hypotheses:
\begin{enumerate}
\item[(A)] \label{A} $|A| \leq \delta^{-\alpha}$.
\item[(B)] \label{B} $|B| \geq \delta^{-\beta}$, and $B$ satisfies the following Frostman condition: 
\begin{displaymath} |B \cap B(x,r)| \leq r^{\kappa}|B|, \qquad \delta \leq r \leq \delta^{\epsilon_{0}}. \end{displaymath} 
\end{enumerate}
Further, let $\nu$ be a Borel probability measure with $\spt (\nu) \subset [0,1]$, and satisfying the Frostman condition $\nu(B(x,r)) \leq r^{\gamma}$ for $x \in \R$ and $0 < r \leq \delta^{\epsilon_{0}}$. Then, there exists $c \in \spt(\nu)$ such that if $B' \subset B$ satisfies $|B'| \geq \delta^{\epsilon}|B|$, then $|A + cB'| \geq \delta^{-\epsilon}|A|$. \end{thm}

\subsection{Reduction to a weaker toy theorem} Even Theorem \ref{mainSubset1} is hard to prove with a direct assault. We will first need to reduce it to an even weaker version. In the statement, we use the following notation (slightly adapted) from He's paper \cite{MR4148151}. Given two sets $A,B \subset [0,1] \cap (\delta \cdot \Z)$, we write
\begin{displaymath} \mathcal{E}(A \mid B,\epsilon) := \{c \in \R : \exists \, B' \subset B \text{ such that } |B'| \geq \delta^{\epsilon}|B| \text{ and } |A + cB'|_{\delta} < \delta^{-\epsilon}|A|\}. \end{displaymath} 

\begin{thm}\label{mainSubset2} Let $0 < \beta \leq \alpha < 1$ and $\kappa,\theta > 0$. Then, for every $\gamma \in ((\alpha - \beta)/(1 - \beta),1]$, there exist $\epsilon_{0},\epsilon,\delta_{0} \in (0,\tfrac{1}{2}]$, depending only on $\alpha,\beta,\gamma,\kappa$, such that the following holds. Let $\delta \in 2^{-\N}$ with $\delta \in (0,\delta_{0}]$, and let $A,B \subset (\delta \cdot \Z) \cap [0,1]$ satisfy the following hypotheses:
\begin{enumerate}
\item[(A)] \label{A} $|A| \leq \delta^{-\alpha}$.
\item[(B)] \label{B} $|B| \geq \delta^{-\beta}$, and $B$ satisfies the following Frostman condition: 
\begin{displaymath} |B \cap B(x,r)| \leq r^{\kappa}|B|, \qquad \delta \leq r \leq \delta^{\epsilon_{0}}. \end{displaymath} 
\end{enumerate}
Further, let $\nu$ be a Borel probability measure with $\spt (\nu) \subset [0,1]$, and satisfying the Frostman condition $\nu(B(x,r)) \leq r^{\gamma}$ for $x \in \R$ and $0 < r \leq \delta^{\epsilon_{0}}$. Then, there exists a subset $B' \subset B$ such that $\nu(\mathcal{E}(A \mid B',\epsilon)) \leq \delta^{\epsilon}$. \end{thm}

I learned this reduction from the paper of He \cite[Proposition 25]{MR4148151}, and his proof works here, up to modifying the notation. The full details are recorded below nonetheless.

\begin{proof}[Proof of Theorem \ref{mainSubset1} assuming Theorem \ref{mainSubset2}] Let $\alpha,\beta,\gamma,\kappa$ be the parameters given in Theorem \ref{mainSubset1}, so that $\gamma > (\alpha - \beta)/(1 - \beta)$. Our task is to find the constants $\epsilon,\epsilon_{0},\delta_{0} \in (0,\tfrac{1}{2}]$, depending only on $\alpha,\beta,\gamma,\kappa$. Start by applying Theorem \ref{mainSubset2} with parameters $\alpha,\bar{\beta},\gamma,\kappa/2$, where $\bar{\beta} < \beta$ is arbitrary with the property that the key inequality
\begin{displaymath} \gamma > (\alpha - \bar{\beta})/(1 - \bar{\beta}) \end{displaymath}
remains valid. Let $\bar{\epsilon},\bar{\epsilon}_{0},\bar{\delta}_{0} \in (0,\tfrac{1}{2}]$ be the constants given by Theorem \ref{mainSubset2}, associated to the parameters $\alpha,\bar{\beta},\gamma,\kappa/2$. We define
\begin{equation}\label{form6} \epsilon_{0} := \bar{\epsilon}_{0} \quad \text{and} \quad \epsilon := \min\left\{\frac{\bar{\epsilon}}{2},\frac{\kappa \bar{\epsilon}_{0}}{8},\frac{\beta - \bar{\beta}}{2} \right\}. \end{equation}
We require that $\delta_{0} \leq \bar{\delta}_{0}$, and there will be a few additional requirements, where for example $\delta \leq \delta_{0}$ needs to be taken small enough relative to the difference $\bar{\epsilon} - \epsilon$. I will not gather these requirements together; they will be pointed out where they appear.

Let $\delta \in 2^{-\N}$ with $\delta \leq \delta_{0}$, and let $A,B,\nu$ be the objects from Theorem \ref{mainSubset1}, satisfying the assumptions of that theorem with constants $\alpha,\beta,\kappa,\gamma$, and $\epsilon_{0},\delta_{0}$ as above. In particular,
\begin{equation}\label{form7} |B| \geq \delta^{-\beta} \quad \text{and} \quad |B \cap B(x,r)| \leq r^{\kappa}|B| \text{ for } x \in \R \text{ and } \delta \leq r \leq \delta^{\epsilon_{0}}. \end{equation}
Evidently $A,B,\nu$ also satisfy the hypotheses of Theorem \ref{mainSubset2} with constants $\alpha,\bar{\beta},\gamma,\kappa/2$, and $\bar{\epsilon}_{0}$. We now perform an "exhaustion" argument to construct a finite sequence of disjoint subsets $B_{1},\ldots,B_{N} \subset B$ with the property
\begin{equation}\label{form20} \nu(\mathcal{E}(A \mid B_{j},\bar{\epsilon})) \leq \delta^{\bar{\epsilon}}, \qquad 1 \leq j \leq N. \end{equation}
Let $B_{1} \subset B$ be the set given initially by Theorem \ref{mainSubset2}. We then assume inductively that we have already constructed disjoint $B_{1},\ldots,B_{n} \subset B$ for some $n \geq 1$. There are two options:
\begin{equation}\label{form1} \Big| B \, \setminus \, \bigcup_{j = 1}^{n} B_{j} \Big| < \delta^{2\epsilon}|B| \quad \text{or} \quad \Big| B \, \setminus \, \bigcup_{j = 1}^{n} B_{j} \Big| \geq \delta^{2\epsilon}|B|.  \end{equation} 
In the former case, the inductive construction terminates, and we define $N := n$. In the latter case, we apply Theorem \ref{mainSubset2} to the objects $A,\nu$, and $B' := B \, \setminus \, \bigcup_{j = 1}^{n} B_{j}$. This is legitimate, because $|B'| \geq \delta^{2\epsilon}|B| \geq \delta^{-\beta - 2\epsilon} \geq \delta^{-\bar{\beta}}$, and
\begin{displaymath} |B' \cap B(x,r)| \stackrel{\eqref{form7}}{\leq} r^{\kappa}|B| \leq \delta^{-2\epsilon}r^{\kappa}|B'| \stackrel{\eqref{form6}}{\leq} r^{\kappa/2}|B'|, \qquad x \in \R, \, \delta \leq r \leq \delta^{\epsilon_{0}} = \delta^{\bar{\epsilon}_{0}}. \end{displaymath}
Therefore $A,B',\nu$ satisfy the hypotheses of Theorem \ref{mainSubset2} with constants $\alpha,\bar{\beta},\kappa/2,\gamma,\bar{\epsilon}_{0}$. Consequently, there exists a further subset $B_{n + 1} \subset B' = B \, \setminus \, \bigcup_{j = 1}^{n} B_{j}$ with the property $\nu(\mathcal{E}(A \mid B_{n + 1},\bar{\epsilon})) \leq \delta^{\bar{\epsilon}}$. This completes the inductive construction of the sequence $B_{1},\ldots,B_{N}$. The construction terminates in $\leq \delta^{-\bar{\epsilon}}$ steps, because the sets $B_{j}$ satisfy $|B_{j}| \geq \delta^{-\bar{\epsilon}}$. Indeed, since $\nu(\mathcal{E}(A \mid B_{j},\bar{\epsilon})) < 1$, there exists $c \in \spt(\nu) \, \setminus \, \mathcal{E}(A \mid B_{j},\bar{\epsilon})$, and then $|A||B_{j}| \geq |A + cB_{j}|_{\delta} \geq \delta^{-\bar{\epsilon}}|A|$.

When the inductive procedure eventually terminates, we write $B_{0} := \bigcup_{j = 1}^{N} B_{j}$. By \eqref{form1}, we have $|B \, \setminus \, B_{0}| < \delta^{2\epsilon}|B|$. Now, note that the claim of Theorem \ref{mainSubset1} is equivalent to proving that there exists a point $c \in \spt(\nu) \, \setminus \, \mathcal{E}(A \mid B,\epsilon)$. We will prove this by showing that $\mathcal{E}(A \mid B,\epsilon)$ has small $\nu$ measure. The first step is to establish the following inclusion:
\begin{equation}\label{form19} \mathcal{E}(A \mid B,\epsilon) \subset \bigcup_{\mathcal{J}} \bigcap_{j \in \mathcal{J}} \mathcal{E}(A \mid B_{j},\bar{\epsilon}), \end{equation}
where the index set $\mathcal{J}$ runs over all subsets of $\{1,\ldots,N\}$ with $\sum_{j \in \mathcal{J}} |B_{j}| \geq \delta^{\epsilon}|B|/4$. The proof is nearly verbatim the same as in \cite[Proposition 25]{MR4148151}, but I record the details here for completeness. If $c \in \mathcal{E}(A \mid B,\epsilon)$, then by definition there exists a subset $B_{c} \subset B$ with $|B_{c}| \geq \delta^{\epsilon}|B|$ and $|A + cB_{c}|_{\delta} < \delta^{-\epsilon}|A|$. Let $\mathcal{J} := \{1 \leq j \leq N : |B_{c} \cap B_{j}| \geq \delta^{\bar{\epsilon}}|B_{j}|\}$. Then $c \in \mathcal{E}(A \mid B_{j},\bar{\epsilon})$ for all $j \in \mathcal{J}$, since $B_{j}' := B_{c} \cap B_{j} \subset B_{j}$ satisfies $|B_{j}'| \geq \delta^{\bar{\epsilon}}|B_{j}|$ and $|A + cB_{j}'|_{\delta} < \delta^{-\epsilon}|A| \leq \delta^{-\bar{\epsilon}}|A|$. This proves \eqref{form19}, once we verify that $\sum_{j \in \mathcal{J}} |B_{j}| \geq \delta^{\epsilon}|B|/4$. 

To see this, recall that $|B \, \setminus \, B_{0}|  \leq \delta^{2\epsilon}|B|$. This implies that $B_{c}$ has large intersection with $B_{0}$ (assuming that $\delta > 0$ is sufficiently small):
\begin{displaymath} |B_{c} \cap B_{0}| \geq \tfrac{1}{2} \cdot \delta^{\epsilon}|B|. \end{displaymath}
Then, if $\delta > 0$ is small enough, and recalling that $\epsilon \leq \bar{\epsilon}/2$, we have
\begin{align*} \tfrac{1}{2} \cdot \delta^{\epsilon}|B| \leq |B_{c} \cap B_{0}| = \sum_{j = 1}^{N} |B_{c} \cap B_{j}| \leq \sum_{j \notin \mathcal{J}} \delta^{\bar{\epsilon}}|B_{j}| + \sum_{j \in \mathcal{J}} |B_{j}| \leq \tfrac{1}{4} \cdot \delta^{\epsilon}|B| + \sum_{j \in \mathcal{J}} |B_{j}|. \end{align*}
Rearranging, $\sum_{j \in \mathcal{J}} |B_{j}| \geq \delta^{\epsilon}|B|/4$. We have now established the inclusion \eqref{form19}.

Finally, it follows from \eqref{form19} and \cite[Lemma 20]{MR4148151} that
\begin{equation}\label{form25} \nu(\mathcal{E}(A \mid B,\epsilon)) \leq \nu \left( \bigcup_{\mathcal{J}} \bigcap_{j \in \mathcal{J}} \mathcal{E}(A \mid B_{j},\bar{\epsilon}) \right) \leq 4\delta^{\bar{\epsilon} - \epsilon} < 1, \end{equation}
assuming once more that $\delta > 0$ is small enough in the final inequality. The proof of \cite[Lemma 20]{MR4148151} is, again, so short that we provide them for the reader's convenience. If $c \in \bigcup_{\mathcal{J}} \bigcap_{j \in \mathcal{J}} \mathcal{E}(A \mid B_{j},\bar{\epsilon})$, then $\sum_{j = 1}^{N} (|B_{j}|/|B|) \cdot \mathbf{1}_{\mathcal{E}(A \mid B_{j},\bar{\epsilon})}(c) \geq \delta^{\epsilon}/4$. Consequently,
\begin{align*} \nu \left( \bigcup_{\mathcal{J}} \bigcap_{j \in \mathcal{J}} \mathcal{E}(A \mid B_{j},\bar{\epsilon}) \right) & \leq 4\delta^{-\epsilon} \sum_{j = 1}^{N} \frac{|B_{j}|}{|B|} \cdot \nu(\mathcal{E}(A \mid B_{j},\bar{\epsilon}))\\
& \leq 4\delta^{-\epsilon} \max_{1 \leq j \leq N} \nu(\mathcal{E}(A \mid B_{j},\bar{\epsilon})) \stackrel{\eqref{form20}}{\leq} 4\delta^{-\epsilon + \bar{\epsilon}}. \end{align*}
This concludes the proof of Theorem \ref{mainSubset1}. \end{proof}

\subsection{Proof of the weaker toy theorem}\label{s:subsetReduction} In this section, we prove Theorem \ref{mainSubset2}. In fact, we only reduce it further to Theorem \ref{mainTechnical}, which eschews the set $\mathcal{E}(A \mid B',\epsilon)$. However, Theorem \ref{mainTechnical} is known, being \cite[Theorem 1.5]{2021arXiv211002779O}, so this reduction will complete the proof of Theorems \ref{mainSubset1}-\ref{mainSubset2}. We start by discussing a few auxiliary results. The first one is the Pl\"unnecke-Ruzsa inequality for different summands: 

\begin{lemma}[Pl\"unnecke-Ruzsa inequality]\label{PRIneq} Let $\delta \in 2^{-\N}$, let $A,B_{1},\ldots,B_{n} \subset \R$ be arbitrary sets, and assume that $|A + B_{i}|_{\delta} \leq K_{i}|A|_{\delta}$ for all $1 \leq i \leq n$, and for some constants $K_{i} \geq 1$. Then, there exists a subset $A' \subset A$ with $|A'|_{\delta} \geq \tfrac{1}{2}|A|_{\delta}$ such that
\begin{displaymath} |A' + B_{1} + \ldots + B_{n}|_{\delta} \lesssim_{n} K_{1}\cdots K_{n} |A'|_{\delta}. \end{displaymath}
\end{lemma}
This form of the inequality is due to Ruzsa \cite{MR2314377}. For a more general result, see \cite[Theorem 1.5]{MR2484645}, by Gyarmati-Matolcsi-Ruzsa. To be accurate, these statements are not formulated in terms of $\delta$-covering numbers, but one may consult \cite[Corollary 3.4]{MR4283564} by Guth-Katz-Zahl to see how to handle the reduction to $\delta$-covering numbers.

The next auxiliary result concerns the existence of \emph{tight} subsets:

\begin{definition} For $\tau,T > 0$ and $N \in \N$, a set $A \subset \delta \cdot \Z$ is called $(\tau,T,N)$-tight if
\begin{displaymath} \max_{1 \leq k \leq N} \frac{|kA|}{|kA'|} \leq T \qquad \text{for all } A' \subset A \text{ with } |A'| \geq \delta^{\tau}|A|. \end{displaymath}
\end{definition}

It will be useful to observe that if $A$ is $(\tau,T,N)$-tight, and $0 < \tau' \leq \tau$, then $A$ is also $(\tau',T,N)$-tight, simply because there are fewer sets $A' \subset A$ to consider.

\begin{lemma}\label{lemma7} Let $\tau > 0$, $N \in \N$, and let $A \subset (\delta \cdot \Z) \cap [0,1]$ be a set with $|A| \geq \delta^{-N^{2}\tau}$. Then, there exists a $(\tau,2\delta^{-1/N},N)$-tight subset $A' \subset A$ of cardinality $|A'| \geq \delta^{N^{2}\tau}|A|$. \end{lemma}

\begin{proof} We find a sequence $A =: A_{0} \supset A_{1} \supset \ldots \supset A_{N^{2}}$ as follows. Assuming that $A_{j - 1}$ has already been selected, and $1 \leq j \leq N^{2}$, we let $A_{j} \subset A_{j - 1}$ be a subset with $|A_{j}| \geq \delta^{\tau}|A_{j - 1}|$ such that the quantity
\begin{displaymath} \max_{1 \leq k \leq N} |kA_{j - 1}|/|kA'| \end{displaymath}
is maximised among all subsets $A' \subset A_{j - 1}$ with $|A'| \geq \delta^{\tau}|A_{j - 1}|$. Thus, we see that if $\max_{1 \leq k \leq N} |kA_{j - 1}|/|kA_{j}| \leq T$, then $A_{j - 1}$ is $(\tau,T,N)$-tight.

Observe that 
\begin{equation}\label{a3}|A_{N^{2}}| \geq \delta^{N^{2}\tau}|A| \geq 1 \quad \text{and} \quad |NA| \leq N\delta^{-1}. \end{equation}
Writing $T := 2\delta^{-1/N}$, we now claim that there exists an index $j \in \{1,\ldots,N^{2}\}$ with
\begin{equation}\label{a2} \max_{1 \leq k \leq N} |kA_{j - 1}|/|kA_{j}| \leq T. \end{equation}
Indeed, if this fails, then by the pigeonhole principle there exists a fixed choice $k \in \{1,\ldots,N\}$, and $n$ indices $j_{1},\ldots,j_{N} \in \{1,\ldots,N^{2}\}$ such that the converse inequality
\begin{displaymath} |kA_{j_{i}}| < T^{-1}|kA_{j_{i} - 1}|, \qquad 1 \leq i \leq N, \end{displaymath}
holds. Since $A_{j} \subset A_{j - 1}$, the inequality $|kA_{j}| \leq |kA_{j - 1}|$ holds for every index $j \in \{1,\ldots,N^{2}\}$, and $|kA_{j_{i}}| < T^{-1}|kA_{j_{i} - 1}|$ for the $n$ special indices $j_{i} \in \{1,\ldots,N^{2}\}$. This forces
\begin{displaymath} 1 \stackrel{\eqref{a3}}{\leq} |kA_{N^{2}}| \leq |kA_{j_{N}}| < T^{-N}|kA| \stackrel{\eqref{a3}}{\leq} (2^{-N}\delta) \cdot (N\delta^{-1}) \leq 1, \end{displaymath}
a contradiction. Now $A' := A_{j - 1} \subset A$, as in \eqref{a2}, is $(\tau,T,N)$-tight, and $|A'| \geq |A_{N^{2}}| \geq \delta^{N^{2}\tau}|A|$. This completes the proof of the lemma. \end{proof}

Finally, we will need the following lemma, which is a $\delta$-discretised version of \cite[Lemma 3.1]{OV18}, or alternatively a version of Bourgain's computations \cite[(7.18)-(7.19)]{Bourgain10} for two different sets (the presence of two different sets adds no difficulties):

\begin{lemma}\label{OVLemma} Let $C_{1},C_{2},C_{3} > 0$, and assume that $A,B \subset \delta \cdot \Z$ are sets with $|A + A| \leq C_{1}|A|$ and $|B + B| \leq C_{2}|B|$. Let moreover $c \in \R$, and let $G \subset A \times B$ be an arbitrary subset with $|G| \geq |A||B|/C_{3}$. Then $|A + cB|_{\delta} \lesssim C_{1}C_{2}C_{3}|\pi_{c}(G)|_{\delta}$.
\end{lemma} 
\begin{proof} Note that
\begin{displaymath} |A + cB|_{\delta} \lesssim \sum_{t \in \delta \cdot \Z} \mathbf{1}_{(A + cB)(\delta)}(t). \end{displaymath}
Fix $t \in (\delta \cdot \Z) \cap (A + cB)(\delta)$ and find $(a,b) \in A \times B$ such that $\dist(t,\pi_{c}(a,b)) \leq \delta$. Then
\begin{displaymath} \dist(t,\pi_{c}(-G + (x,y))) \leq \dist(t,\pi_{c}(a,b)) \leq \delta, \qquad (x,y) \in G + (a,b). \end{displaymath}
Moreover, any candidates $(x,y) \in G + (a,b)$ satisfy
\begin{displaymath} (x,y) \in G + (a,b) \subset (A \times B) + (A \times B) = (A + A) \times (B + B), \end{displaymath}
so there are $\geq |G + (a,b)| = |G|$ points $(x,y) \in (A + A) \times (B + B)$ with the property $\dist(t,\pi_{c}(-G + (x,y))) \leq \delta$. It follows that
\begin{displaymath} \mathbf{1}_{(A + cB)(\delta)}(t) = 1 \leq \frac{1}{|G|} \sum_{(x,y) \in (A + A) \times (B + B)} \mathbf{1}_{\pi_{c}(-G + (x,y))(\delta)}(t). \end{displaymath}
Since $|\pi_{c}(G)|_{\delta} \sim |\pi_{c}(-G + (x,y))|_{\delta}$ for every $(x,y) \in (\delta \cdot \Z)^{2}$, we have
\begin{align*} |A + cB|_{\delta} & \lesssim \frac{1}{|G|} \sum_{(x,y) \in (A + A) \times (B + B)} \sum_{t \in \delta \cdot \Z} \mathbf{1}_{\pi_{c}(-G + (x,y))(\delta)}(t)\\
& \lesssim \frac{|A + A||B + B||\pi_{c}(G)|_{\delta}}{|G|} \leq C_{1}C_{2}C_{3}|\pi_{c}(G)|_{\delta}, \end{align*}
as claimed. \end{proof}

We are now ready to carry out the main task in this section, namely reducing the proof of Theorem \ref{mainSubset2} to the following result (which is \cite[Theorem 1.5]{2021arXiv211002779O}):

\begin{thm}\label{mainTechnical} Let $0 < \beta \leq \alpha < 1$ and $\kappa > 0$. Then, for every $\gamma \in ((\alpha - \beta)/(1 - \beta),1]$, there exist $\epsilon_{0},\epsilon,\delta_{0} \in (0,\tfrac{1}{2}]$, depending only on $\alpha,\beta,\gamma,\kappa$, such that the following holds. Let $\delta \in 2^{-\N}$ with $\delta \in (0,\delta_{0}]$, and let $A,B \subset (\delta \cdot \Z) \cap [0,1]$ satisfy the following hypotheses:
\begin{enumerate}
\item[(A)] \label{A} $|A| \leq \delta^{-\alpha}$.
\item[(B)] \label{B} $|B| \geq \delta^{-\beta}$, and $B$ satisfies the following Frostman condition: 
\begin{displaymath} |B \cap B(x,r)| \leq r^{\kappa}|B|, \qquad \delta \leq r \leq \delta^{\epsilon_{0}}. \end{displaymath} 
\end{enumerate}
Further, let $\nu$ be a Borel probability measure with $\spt (\nu) \subset [0,1]$, and satisfying the Frostman condition $\nu(B(x,r)) \leq r^{\gamma}$ for $x \in \R$ and $\delta \leq r \leq \delta^{\epsilon_{0}}$. Then, there exists a point $c \in \spt (\nu)$ such that
\begin{displaymath} |A + cB|_{\delta} \geq \delta^{-\epsilon}|A|. \end{displaymath}
\end{thm}

\begin{proof}[Proof of Theorem \ref{mainSubset2} assuming Theorem \ref{mainTechnical}] Fix the parameters $0 < \beta \leq \alpha < 1$, $\kappa > 0$, and $\gamma \in ((\alpha - \beta)/(1 - \beta),1]$, as in Theorem \ref{mainSubset2}. Start by applying Theorem \ref{mainTechnical} with the following slightly modified parameters:
\begin{displaymath} 0 < \bar{\beta} \leq \bar{\alpha} < 1, \, \tfrac{\kappa}{2} > 0, \quad \text{and} \quad \bar{\gamma} > (\bar{\alpha} - \bar{\beta})/(1 - \bar{\beta}), \end{displaymath}
where $\bar{\alpha} > \alpha$ and $\bar{\beta} < \beta$ and $\bar{\gamma} < \gamma$ are arbitrary choices such that the final inequality is valid. The parameters $\bar{\alpha},\bar{\beta},\bar{\gamma}$ should be viewed as functions of $\alpha,\beta,\gamma$ (we leave finding explicit expressions to the reader): therefore, any future dependence on $\bar{\alpha},\bar{\beta},\bar{\gamma}$ will, in fact, be a dependence on $\alpha,\beta,\gamma$. Then, let 
\begin{displaymath} \bar{\epsilon}_{0},\bar{\epsilon},\bar{\delta}_{0} > 0 \end{displaymath}
be the constants given by Theorem \ref{mainTechnical}, which only depend on $\bar{\alpha},\bar{\beta},\kappa/2,\bar{\gamma}$, and such that the conclusion of Theorem \ref{mainTechnical} is valid. Our task is to find constants $\epsilon,\epsilon_{0},\delta_{0} > 0$, which may depend on all of the constants $\alpha,\bar{\alpha},\beta,\bar{\beta},\gamma,\bar{\gamma},\kappa,\bar{\epsilon}_{0},\bar{\epsilon},\bar{\delta}_{0}$, such that Theorem \ref{mainSubset2} is valid with constants $\alpha,\beta,\gamma,\kappa$. The choice of $\epsilon_{0}$ is particularly simple:
\begin{equation}\label{a22} \epsilon_{0} = \bar{\epsilon}_{0}. \end{equation}
For $\delta_{0}$, we will need that $\delta_{0} \leq \bar{\delta}_{0}$, and there will be an additional dependence on $\bar{\alpha},\bar{\epsilon}$, which will be clarified during the proofs of \eqref{a5} and \eqref{a19}. To define the constant $\epsilon \in (0,\tfrac{1}{2}]$, we first introduce an auxiliary natural number $N \in \N$ satisfying
\begin{equation}\label{a23} \frac{3}{N} + \frac{1}{\log_{2} N} \leq \bar{\epsilon}/2. \end{equation}
Then, we choose $\epsilon > 0$ so small that
\begin{equation}\label{a14} N^{4N + 1}\epsilon \leq \min\left\{\bar{\epsilon}_{0}\kappa/2,\beta - \bar{\beta},\bar{\alpha} - \alpha\right\} \quad \text{and} \quad \epsilon \leq \bar{\epsilon}_{0}(\gamma - \bar{\gamma}). \end{equation}
We now claim that Theorem \ref{mainSubset2} holds with the constants $\epsilon,\epsilon_{0},\delta_{0}$ (given the parameters $\alpha,\beta,\gamma,\kappa$). Let $\delta \in 2^{-\N}$ with $\delta \leq \delta_{0}$, and let $A,B,\nu$ be objects satisfying the hypotheses of Theorem \ref{mainSubset2} with constants $\alpha,\beta,\gamma,\kappa,\epsilon_{0}$. Thus $A,B \subset [0,1] \cap (\delta \cdot \Z)$, $|A| \leq \delta^{-\alpha}$, and $\nu(B(x,r)) \leq r^{\gamma}$ for $x \in \R$ and $\delta \leq r \leq \delta^{\epsilon_{0}} = \delta^{\bar{\epsilon}_{0}}$. Further, $|B| \geq \delta^{-\beta}$, and
\begin{equation}\label{form17} |B \cap B(x,r)| \leq r^{\gamma}|B|, \qquad x \in \R, \, \delta \leq r \leq \delta^{\epsilon_{0}} = \delta^{\bar{\epsilon}_{0}}. \end{equation}
The claim is that there exists a subset $B' \subset B$ such that $\nu(\mathcal{E}(A \mid B',\epsilon)) \leq \delta^{\epsilon}$. We proceed by making a counter assumption:

\begin{counter} $\nu(\mathcal{E}(A \mid B',\epsilon)) > \delta^{\epsilon}$ for all $B' \subset B$. \end{counter}

We will use our \textbf{Counter assumption} and Lemmas \ref{lemma7} and \ref{OVLemma} to construct a sequence $\{H_{n}\}_{n = 1}^{N} \subset \delta \cdot \Z$ with $|H_{n}| \leq \delta^{-\bar{\alpha}}$. The point will be, omitting all technical details, that once this sequence has been constructed, we will find an index $n \in \{0,\ldots,N - 1\}$ with the property that $|H_{n} + cB| < \delta^{-\bar{\epsilon}}|H_{n}|$ for all $c \in \spt(\nu)$. This (or the more technical version of it) will violate Theorem \ref{mainTechnical}, and show that the \textbf{Counter assumption} is false.

Let $\{\tau_{n}\}_{n = 0}^{N}$ be the finite decreasing sequence
\begin{equation}\label{defTau} \tau_{n} := N^{4N - 3n}\epsilon, \qquad 0 \leq n \leq N. \end{equation}
While we construct the sets $H_{n}$, we will simultaneously find elements $c_{1},c_{2},\ldots,c_{N} \in C = \spt(\nu)$, subsets $C_{n} \subset C$ of measure $\nu(C_{n}) \geq \delta^{\epsilon}$, and a decreasing sequence $B =: B_{0} \supset B_{1} \supset \ldots \supset B_{N}$ with the following three properties:
\begin{enumerate}
\item $|B_{n + 1}| \geq \delta^{\tau_{n}}|B_{n}|$ for $0 \leq n \leq N - 1$,
\item $|A + c_{n}B_{n}| < \delta^{-\epsilon}|A|$ for $1 \leq n \leq N$,
\item $B_{n}$ is $(\tau_{n},2\delta^{1/N},N)$-tight for $1 \leq n \leq N$.
\end{enumerate}
In particular, it follows from property (1) that
\begin{equation}\label{a13} |B_{n}| \geq \delta^{N\tau_{0}}|B| \geq \delta^{N^{4N + 1}\epsilon}|B|, \qquad 0 \leq n \leq N. \end{equation} 
To initialise the definition of the sets $B_{n},C_{n},H_{n}$, and the elements $c_{n} \in C$, set $B_{0} := B$ and $H_{0} := \emptyset$. (the properties (2)-(3) do not concern the case $n = 0$). Assume that the sets $B_{n}$ have already been constructed for some $0 \leq n \leq N - 1$, and recall the notation $(H)_{\delta} := (\delta \cdot \Z) \cap H(\delta)$ for arbitrary $H \subset \R$. By the \textbf{Counter assumption} applied to the set $B' := B_{n} \subset B$, there now corresponds a subset 
\begin{equation}\label{form21} C_{n + 1} := \mathcal{E}(A \mid B_{n},\epsilon) \subset \spt (\nu) \end{equation}
of measure $\nu(C_{n + 1}) \geq \delta^{\epsilon}$ with the property that for all $c \in C_{n + 1}$, there is a further subset $\bar{B}_{c} \subset B_{n}$ of cardinality $|\bar{B}_{c}| \geq \delta^{\epsilon}|B_{n}|$ such that $|A + c\bar{B}_{c}| < \delta^{-\epsilon}|A|$. We will define $H_{n + 1}$ as either $H_{n + 1} := H_{n} + H_{n} \subset \delta \cdot \Z$, or 
\begin{equation}\label{a11} H_{n + 1} := (H_{n} + c \cdot NB_{c})_{\delta} \subset \delta \cdot \Z, \end{equation}
where $c \in C_{n + 1}$, and $B_{c} \subset \bar{B}_{c} \subset B_{n}$ is a certain set satisfying the constraints (1)-(3). It turns out that subsets of this kind exist for all $c \in C_{n + 1}$: this will be proved shortly, but should be taken for granted for now. For every $c \in C_{n + 1}$, we then pick the subset $B_{c} \subset \bar{B}_{c} \subset B_{n}$ which satisfies (1)-(3), and maximises the number $|H_{n} + c \cdot NB_{c}|_{\delta}$, among all possible $c \in C_{n + 1}$, and subsets $B_{c} \subset \bar{B}_{c}$ satisfying (1)-(3). Once the optimal $c \in C_{n + 1}$ and $B_{c} \subset \bar{B}_{c} \subset B_{n}$ have been located, we finally check if 
\begin{displaymath} |H_{n} + c \cdot NB_{c}|_{\delta} \geq |H_{n} + H_{n}|. \end{displaymath}
If this happens, then $H_{n + 1}$ is defined as in \eqref{a11}. Otherwise $H_{n + 1} := H_{n} + H_{n}$. Note that in both cases
\begin{equation}\label{a12} |H_{n} + H_{n}| \leq |H_{n + 1}|. \end{equation}
If $H_{n}$ was defined by \eqref{a11}, for some 
\begin{displaymath} c_{n + 1} := c \in C_{n + 1}, \end{displaymath}
then we set $B_{n + 1} := B_{c} \subset B_{n}$, where $B_{c}$ is the maximising set found above. If $H_{n + 1} = H_{n} + H_{n}$, we simply define $B_{n + 1} := B_{n}$, and $c_{n + 1} := c_{n}$. Note that in all cases the properties (1)-(3) are satisfied, and $B_{n + 1}$ is $(\tau_{n + 1},2\delta^{1/N},N)$-tight. This is even true if $B_{n + 1}$ was defined via the "second scenario" as $B_{n + 1} = B_{n}$: indeed, since $H_{0} = \emptyset$, this is only possible if $n \geq 1$, and then we already know that $B_{n}$ is $(\tau_{n},2\delta^{1/N},N)$-tight. Then $B_{n + 1} = B_{n}$ is also $(\tau_{n + 1},2\delta^{1/N},N)$-tight simply because $\tau_{n + 1} \leq \tau_{n}$.

This completes the inductive definition of the sets $B_{n},C_{n},H_{n}$, and elements $c_{n} \in C$, for $1 \leq n \leq N$. Note that $H_{n} \subset (\delta \cdot \Z) \cap [0,N^{n}]$ by a straightforward induction, so $|H_{n}| \leq N^{n} \delta^{-1}$. Therefore, by the pigeonhole principle, there exists an index $n \in \{0,\ldots,N - 1\}$ such that
\begin{equation}\label{a4} |H_{n} + H_{n}| \stackrel{\eqref{a12}}{\leq} |H_{n + 1}| \leq (N^{n}\delta^{-1})^{1/N}|H_{n}| \leq N\delta^{-1/N}|H_{n}|. \end{equation}
Since $H_{0} = \emptyset \neq H_{1}$, the middle inequality cannot be satisfied with $n = 0$, and we see that actually $n \in \{1,\ldots,N - 1\}$. We now claim that, for this particular index $n$, fixed for the remainder of the argument, it holds that
\begin{equation}\label{a5} |H_{n} + cB_{n}|_{\delta} < \delta^{-\bar{\epsilon}}|H_{n}|, \qquad c \in C_{n + 1}, \end{equation}
assuming that $\delta \leq \delta_{0}$, and the threshold $\delta_{0} > 0$ is sufficiently small, depending only on $N$ (hence "$\bar{\epsilon}$" by our choice \eqref{a23}). To see this, we first record that $2^{k}B_{n} \subset 2^{N}B \subset (\delta \cdot \Z) \cap [0,2^{N}]$ for all $1 \leq k \leq N$, so by another application of the pigeonhole principle, there exists an index $0 \leq k \leq \log_{2} (N - 1)$ such that
\begin{equation}\label{a8} |2^{k + 1}B_{n}| \leq (2^{N}\delta^{-1})^{1/\log_{2} N}|2^{k}B_{n}|. \end{equation} 
We also fix this index $k \in \{0,\ldots,\log_{2}(N - 1)\}$ for the remainder of the argument.

Now, to prove \eqref{a5}, fix $c \in C_{n + 1} \subset [0,1]$, and recall the subset $\bar{B}_{c} \subset B_{n}$ defined right below \eqref{form21}, satisfying $|\bar{B}_{c}| \geq \delta^{\epsilon}|B_{n}|$ and $|A + c\bar{B}_{c}| < \delta^{-\epsilon}|A|$. We use Lemma \ref{lemma7} to find a $(\tau_{n + 1},2\delta^{-1/N},N)$-tight subset $B_{c} \subset \bar{B}_{c} \subset B_{n}$ of cardinality 
\begin{align} |B_{c}| \geq \delta^{N^{2}\tau_{n + 1}}|\bar{B}_{c}| & \stackrel{\eqref{defTau}}{\geq} \delta^{N^{2}N^{4N - 3(n + 1)}\epsilon + \epsilon}|B_{n}| \notag\\
& \,\,\, \geq \delta^{(N^{4N - 3n - 1} + 1)\epsilon}|B_{n}|\notag\\
&\label{a7} \,\,\, \geq \delta^{N^{4N - 3n}\epsilon}|B_{n}| = \delta^{\tau_{n}}|B_{n}|.\end{align}
We used the elementary inequality $N^{4N - 3n - 1} + 1 \leq N^{4N - 3n}$, for $N \geq 2$ and $0 \leq n \leq N$.

A combination of \eqref{a7}, the tightness of $B_{c}$, and the inequality $|A + cB_{c}| < \delta^{-\epsilon}|A|$, shows that $B_{c} \subset \bar{B}_{c} \subset B_{n}$ satisfies all the requirements (1)-(3), and is therefore a competitor in the definition of $H_{n + 1}$. In particular, now we have shown, as promised, that such competitors exist for all $c \in C_{n + 1}$. Moreover, since $2^{k} \leq N$ (as in \eqref{a8}), it follows that
\begin{equation}\label{a15} |H_{n} + c \cdot 2^{k}B_{c}|_{\delta} \lesssim |H_{n} + c \cdot NB_{c}|_{\delta} \lesssim |H_{n + 1}| \stackrel{\eqref{a4}}{\leq} N\delta^{-1/N} |H_{n}|. \end{equation}
With this bound in hand, we continue to estimate as follows, applying Lemma \ref{OVLemma} to the sets $H_{n},2^{k}B_{n} \subset \delta \cdot \Z$, and the subset $G = H_{n} \times 2^{k}B_{c} \subset H_{n} \times 2^{k}B_{n}$ which satisfies $|G| = |H_{n}||2^{k}B_{n}|/C_{3}$ with constant $C_{3} = |2^{k}B_{n}|/|2^{k}B_{c}|$: 
\begin{align} |H_{n} + c B_{n}|_{\delta} & \lesssim |H_{n} + c \cdot 2^{k}B_{n}|_{\delta} \notag \\
&\label{a9} \lesssim \frac{|H_{n} + H_{n}|}{|H_{n}|} \cdot \frac{|2^{k + 1}B_{n}|}{|2^{k}B_{n}|} \cdot \frac{|2^{k}B_{n}|}{|2^{k}B_{c}|} \cdot |H_{n} + c \cdot 2^{k}B_{c}|_{\delta}.  \end{align} 
Apart from \eqref{a15}, the individual factors are bounded from above as follows:
\begin{itemize}
\item $|H_{n} + H_{n}|/|H_{n}| \leq N\delta^{-1/N}$ by \eqref{a4},
\item $|2^{k + 1}B_{n}|/|2^{k}B_{n}| \leq (2^{N}\delta^{-1})^{1/\log_{2} N}$ by \eqref{a8},
\item $|2^{k}B_{n}|/|2^{k}B_{c}| \leq 2\delta^{-1/N}$ by the $(\tau_{n},2\delta^{-1/N},N)$-tightness of $B_{n}$, and by \eqref{a7}.
\end{itemize}
Plugging these estimates into \eqref{a9} yields
\begin{equation}\label{form23} |H_{n} + cB_{n}|_{\delta} \lesssim_{N} \delta^{-3/N - 1/\log_{2} N }|H_{n}| \stackrel{\eqref{a23}}{\leq} \delta^{-\bar{\epsilon}/2}|H_{n}|, \qquad c \in C_{n + 1}. \end{equation}
This completes the proof of \eqref{a5}, if $\delta > 0$ is small enough depending on $\bar{\epsilon},N$, both of which only depend on $\alpha,\beta,\gamma,\kappa$.

We next plan to use \eqref{a5} to contradict Theorem \ref{mainTechnical} with parameters $\bar{\alpha},\bar{\beta},\bar{\gamma},\kappa/2$, and the objects $H_{n},B_{n}$, and $\bar{\nu} = \nu(C_{n + 1})^{-1}\nu|_{C_{n + 1}}$. The first task it to use the Pl\"unnecke-Ruzsa inequality, Lemma \ref{PRIneq}, to show
\begin{equation}\label{a19} |H_{n}| \leq \delta^{-\bar{\alpha}}, \end{equation}
assuming that $\delta > 0$ is sufficiently small in terms of $N,\bar{\alpha}$. Indeed, note that $H_{n}$ can be written as a sum of $\leq N^{n} \leq N^{N}$ sets of the form $c_{m}B_{m}$, for some $1 \leq m \leq n$. Each of these sets individually satisfies $|A + c_{m}B_{m}|_{\delta} < \delta^{-\epsilon}|A|$. We may therefore infer that
\begin{displaymath} |H_{n}| \lesssim_{N} \delta^{-2N^{N}\epsilon}|A| \leq \delta^{-2N^{N}\epsilon - \alpha}. \end{displaymath}
from Lemma \ref{PRIneq}. This inequality implies \eqref{a19} for small enough $\delta > 0$, recalling our choice of constants at \eqref{a14}.

Recall from \eqref{form17} that the set $B$ satisfies a Frostman condition with exponent $\kappa$:
\begin{displaymath} |B \cap B(x,r)| \leq r^{\kappa}|B|, \qquad x \in \R, \, \delta \leq r \leq \delta^{\epsilon_{0}}. \end{displaymath}
Since $B_{n} \subset B$, and $|B_{n}| \geq \delta^{N^{4N + 1}\epsilon}|B|$ by \eqref{a13} we deduce that $B_{n}$ satisfies a Frostman condition with parameters $\bar{\epsilon}_{0} = \epsilon_{0}$ (recall \eqref{a22}) and $\kappa/2$:
\begin{equation}\label{a24} |B_{n} \cap B(x,r)| \leq \delta^{-N^{4N + 1}\epsilon}r^{\kappa}|B_{n}| \stackrel{\eqref{a14}}{\leq} r^{\kappa/2}|B_{n}|, \qquad x \in \R, \, \delta \leq r \leq \delta^{\bar{\epsilon}_{0}}. \end{equation}
Moreover, since $|B| \geq \delta^{-\beta}$ by assumption (see above \eqref{form17}), we have
\begin{equation}\label{a25} |B_{n}| \geq \delta^{N^{4N + 1}\epsilon}|B| \stackrel{\eqref{a14}}{\geq} \delta^{-\bar{\beta}}. \end{equation}
Finally, we verify that the probability measure $\bar{\nu} := \nu(C_{n + 1})^{-1} \cdot \nu|_{C_{n + 1}}$ satisfies a Frostman condition with exponent $\bar{\gamma}$. Since $\nu$ itself satisfies the Frostman condition $\nu(B(x,r)) \leq r^{\gamma}$ for all $\delta \leq r \leq \delta^{\epsilon_{0}} = \delta^{\bar{\epsilon}_{0}}$, and $\nu(C_{n + 1}) \geq \delta^{\epsilon}$, we see that
\begin{displaymath} \bar{\nu}(B(x,r)) \leq \delta^{-\epsilon}\nu(B(x,r)) \leq \delta^{-\epsilon}r^{\gamma} \stackrel{\eqref{a14}}{\leq} r^{\bar{\gamma}}, \qquad x \in \R, \, \delta \leq r \leq \delta^{\bar{\epsilon}_{0}}. \end{displaymath} 
We have now reached a situation which violates Theorem \ref{mainTechnical} for the choice of parameters $\bar{\alpha},\bar{\beta},\kappa/2,\bar{\gamma}$. The objects $H_{n},B_{n},\bar{\nu}$ satisfy all the hypotheses by \eqref{a19}-\eqref{a25}, but nevertheless $|H_{n} + cB_{n}|_{\delta} < \delta^{-\bar{\epsilon}}|H_{n}|$ for all $c \in C_{n + 1}$ according to \eqref{a5}, where $C_{n + 1}$ is a set of full $\bar{\nu}$ measure. Therefore the \textbf{Counter assumption} is false, and the proof of Theorem \ref{mainSubset2} is complete.

To be precise, we have cut one corner: $H_{n}$ may not be a subset of $[0,1]$: we only know that $H_{n} \subset [0,N^{n}] \subset [0,N^{N}]$. However, one can easily fix this by picking the most $H_{n}$-populous unit interval $[r,r + 1] \subset [0,N^{n}]$, which contains $\geq N^{-N}|H_{n}| \geq \delta^{\bar{\epsilon}/4}|H_{n}|$ points of $H_{n}$ if $\delta > 0$ is small enough, and replacing $H_{n}$ by $\bar{H}_{n} := H_{n} \cap [r,r + 1] - \{r\} \subset [0,1]$. After replacing $H_{n}$ with $\bar{H}_{n}$, the estimate \eqref{form23} remains valid with constant $\delta^{-3\bar{\epsilon}/4}$ instead of $\delta^{-\bar{\epsilon}/2}$. This is still good enough to imply \eqref{a5}. \end{proof}

\section{Proof of the main theorem}\label{s:mainProof}

In this section, we will prove Theorem \ref{main} by reducing it to its toy version, Theorem \ref{mainSubset1}. We will need the asymmetric Balog-Szemer\'edi-Gowers theorem, see the book of Tao and Vu, \cite[Theorem 2.35]{MR2289012}. We state the result in the following slightly weaker form (following Shmerkin's paper \cite[Theorem 3.2]{Sh}):
\begin{thm}[Asymmetric Balog-Szemer\'edi-Gowers theorem]\label{BSG} Given $\zeta > 0$, there exists $\xi > 0$ such that the following holds for $\delta \in 2^{-\N}$ small enough. Let $A,B \subset (\delta \cdot \Z) \cap [0,1]$ be finite sets, and assume that there exist $c \in [\tfrac{1}{2},1]$ and $G \subset A \times B$ satisfying 
\begin{equation}\label{bsg1} |G| \geq \delta^{\xi}|A||B| \quad \text{and} \quad |\{x + cy : (x,y) \in G\}|_{\delta} = |\pi_{c}(G)|_{\delta} \leq \delta^{-\xi}|A|. \end{equation}
Then there exist subsets $A' \subset A$ and $B' \subset B$ with the properties
\begin{equation}\label{bsg2} |A'||B'| \geq \delta^{\zeta}|A||B| \quad \text{and} \quad |A' + cB'|_{\delta} \leq \delta^{-\zeta}|A|. \end{equation}
\end{thm}

\begin{remark} In the references for Theorem \ref{BSG} cited above, the assumption $|\pi_{c}(G)|_{\delta} \leq \delta^{-\xi}|A|$ in \eqref{bsg1} is replaced by $|\pi_{1}(G)| \leq \delta^{-\xi}|A|$, and the conclusion \eqref{bsg2} is replaced by $|A' + B'| \leq \delta^{-\zeta}|A|$. For $c \in [\tfrac{1}{2},1]$, it is easy to see that the two variants of the theorem are formally equivalent. The details are left to the reader. The idea is to begin by applying the standard version of Theorem \ref{BSG} to the sets $B_{c} := (cB)_{\delta} \subset \delta \cdot \Z$ and $G_{c} := \{(x,(cy)_{\delta}) : (x,y) \in G\} \subset A \times B_{c}$, which satisfy $|B_{c}| \sim |B|$, $|G_{c}| \sim |G|$, and $|\pi_{1}(G_{c})| \lesssim \delta^{-\xi}|A|$. \end{remark}

\begin{proof}[Proof of Theorem \ref{main} assuming Theorem \ref{mainSubset1}] Let $\alpha,\beta,\gamma,\kappa$ be the constants for which we are supposed to prove Theorem \ref{main}. Thus $\gamma > (\alpha - \beta)/(1 - \beta)$. Our task is to find the constants $\epsilon,\epsilon_{0},\delta_{0} \in (0,\tfrac{1}{2}]$ such that the conclusion of Theorem \ref{main} holds. To this end, pick $\bar{\alpha} > \alpha$, $\bar{\beta} < \beta$, and $\bar{\gamma} < \gamma$ in such a way that the key inequality
\begin{displaymath} \bar{\gamma} > (\bar{\alpha} - \bar{\beta})/(1 - \bar{\beta}) \end{displaymath}
persists. This can be done explicitly in such a way that $\bar{\alpha},\bar{\beta},\bar{\gamma}$ are functions of $\alpha,\beta,\gamma$: therefore, any future dependence on $\bar{\alpha},\bar{\beta},\bar{\gamma}$ will, in fact, be a dependence on $\alpha,\beta,\gamma$.

Let $\bar{\epsilon},\bar{\epsilon}_{0},\bar{\delta}_{0}$ be the constants given by Theorem \ref{mainSubset1} applied with parameters $\bar{\alpha},\bar{\beta},\bar{\gamma},\kappa/2$. Thus, $\bar{\epsilon},\bar{\epsilon}_{0},\bar{\delta}_{0}$ are eventually functions of $\alpha,\beta,\gamma,\kappa$. We define $\epsilon,\epsilon_{0},\delta_{0}$ based on $\bar{\epsilon},\bar{\epsilon}_{0},\bar{\delta}_{0}$. First, we set $\epsilon_{0} := \bar{\epsilon}_{0}$. We also fix $\delta_{0} \in (0,\bar{\delta}_{0}]$. There will be a few additional requirements on $\delta_{0}$, depending on $\alpha,\beta,\gamma,\kappa$ only. These will be clarified when they arise. We then finally determine the constant $\epsilon$. First, we fix a natural number $N \sim 1/\bar{\epsilon}$, sufficiently large that the following holds:
\begin{equation}\label{b23} (N - 1)^{-1} < \bar{\epsilon}/2. \end{equation}
Then, we fix the auxiliary constant
\begin{equation}\label{b9} \zeta := \min\left\{\frac{\bar{\epsilon}}{20 N},\frac{\bar{\epsilon}_{0}\kappa}{4N}, \frac{\bar{\alpha} - \alpha}{2N(N + 1)},\frac{\beta - \bar{\beta}}{2N},\frac{\epsilon_{0}(\gamma - \bar{\gamma})}{2N} \right\}. \end{equation}
Now, we let $\epsilon := \xi(\zeta) > 0$ be the constant given by the Balog-Szemerer\'edi-Gowers theorem applied with the constant $\zeta > 0$ from \eqref{b9}. This means that if $G \subset A \times B$ satisfies $|G| \geq \delta^{\epsilon}|A||B|$ and $|\pi_{c}(G)|_{\delta} \leq \delta^{-\epsilon}|A|$, then there exist $A' \subset A$ and $B' \subset B$ as in \eqref{bsg2}. 

Armed with these choices of parameters, we are prepared to prove Theorem \ref{main}. Fix $\delta \in 2^{-\N}$ with $\delta \leq \delta_{0}$, and let $A,B,\nu$ be a triple satisfying the hypotheses of Theorem \ref{main} with constants $\alpha,\beta,\gamma,\kappa$. In particular, $|A| \leq \delta^{\alpha}$, and $|B| \geq \delta^{-\beta}$, and 
\begin{equation}\label{form10} |B \cap B(x,r)| \leq r^{\kappa}|B|, \qquad x \in \R, \, \delta \leq r \leq \delta^{\epsilon_{0}} = \delta^{\bar{\epsilon}_{0}}. \end{equation}
Also, recall that $\nu$ is a probability measure on $[\tfrac{1}{2},1]$ satisfying $\nu(B(x,r)) \leq r^{\gamma}$ for all $\delta \leq r \leq \delta^{\epsilon_{0}}$. We claim that there exists $c \in C := \spt(\nu)$ such that whenever $G \subset A \times B$ is a subset with $|G| \geq \delta^{\epsilon}|A||B|$, then $|\pi_{c}(G)|_{\delta} \geq \delta^{-\epsilon}|A|$. 

We make a counter assumption: the property above fails for every $c \in C$. Then, by the choice $\epsilon = \xi(\zeta)$, and Theorem \ref{BSG}, for every $c \in C$ there exist subsets $A_{c} \subset A$ and $B_{c} \subset B$, for every $c \in C$, with the properties 
\begin{equation}\label{b5} |A_{c} \times B_{c}| \geq \delta^{\zeta}|A||B| \quad \text{and} \quad |A_{c} + cB_{c}|_{\delta} \leq \delta^{-\zeta}|A|. \end{equation}
We observe that
\begin{displaymath} \int \ldots \int |(A_{c_{1}} \times B_{c_{1}}) \cap \ldots \cap (A_{c_{N}} \times B_{c_{N}})| \, d\nu(c_{1})\cdots d\nu(c_{N}) \geq \delta^{N \zeta}|A||B| \end{displaymath}
by H\"older's inequality. Using $(A \times B) \cap (C \times D) = (A \cap C) \times (B \cap D)$, and Chebyshev's inequality, and $\nu(\R) = 1$, it follows that the set
\begin{equation}\label{def:omega} \Omega := \{(c_{1},\ldots,c_{N}) \in C^{N} : |(A_{c_{1}} \cap \ldots \cap A_{c_{N}}) \times (B_{c_{1}} \cap \ldots \cap B_{c_{N}})| \geq \tfrac{1}{2}\delta^{N\zeta}|A||B|\} \end{equation}
satisfies
\begin{equation}\label{form12} \nu^{N}(\Omega) \geq \tfrac{1}{2} \cdot \delta^{N\zeta} \end{equation}
For $c_{1},\ldots,c_{n} \in C$ fixed, we define
\begin{displaymath} \Omega_{c_{1}\cdots c_{n}} := \{(c_{n + 1},\ldots,c_{N}) \in C^{N - n} : (c_{1},\ldots,c_{N}) \in \Omega\}. \end{displaymath}
It follows easily from Fubini's theorem that
\begin{equation}\label{b6} \nu^{N - n}(\Omega_{c_{1}\cdots c_{n}}) = \int \nu^{N - n - 1}(\Omega_{c_{1}\cdots c_{n}c}) \, d\nu(c) \end{equation}
for all $c_{1},\ldots,c_{n} \in C$, and $1 \leq n \leq N - 2$. The same remains true for $n = 0$, if the left hand side is interpreted as $\nu^{N}(\Omega)$, and $c_{1}\cdots c_{n}c = c$. Equation \eqref{b6} also remains valid for $n = N - 1$ if we define the notation $\nu^{N - n - 1} = \nu^{0}$ as follows:
\begin{equation}\label{b7} \nu^{0}(\Omega_{c_{1}\cdots c_{N - 1}c}) := \mathbf{1}_{\Omega}(c_{1},\ldots,c_{N - 1},c). \end{equation}
We will use this notation in the sequel. 

For $(c_{1},\ldots,c_{N}) \in C^{N}$ fixed, we define decreasing sequences of sets $\{A_{c_{1}\cdots c_{n}}\}_{n = 1}^{N}$ and $\{B_{c_{1}\cdots c_{n}}\}_{n = 1}^{N}$ as follows:
\begin{equation}\label{form14} A_{c_{1}\cdots c_{n}} := A_{c_{1}} \cap \ldots \cap A_{c_{n}} \quad \text{and} \quad B_{c_{1}\cdots c_{n}} := B_{c_{1}} \cap \ldots \cap B_{c_{n}}, \quad 1 \leq n \leq N. \end{equation}
The definition formally makes sense for $(c_{1},\ldots,c_{N}) \in C^{N}$, but will only be useful for $(c_{1},\ldots,c_{N}) \in \Omega$. Namely, if $(c_{1},\ldots,c_{N}) \in \Omega$, then it follows from the definition \eqref{def:omega} that
\begin{equation}\label{b4} |A_{c_{1}\cdots c_{n}}| \geq |A_{c_{1}\cdots c_{N}}| \geq \tfrac{1}{2} \cdot \delta^{N\zeta}|A| \quad \text{and} \quad |B_{c_{1}\cdots c_{n}}| \geq \tfrac{1}{2} \cdot \delta^{N\zeta}|B|. \end{equation}
We now construct the sets $\{H_{n}\}_{n = 1}^{N} \subset \delta \cdot \Z$. At the same time, we will construct subsets $C_{1},\ldots,C_{N} \subset C$, and points $c_{n} \in C_{n}$, $1 \leq n \leq N$, with the properties  
\begin{equation}\label{b3} \nu^{N - n}(\Omega_{c_{1}\cdots c_{n}}) \geq 2^{-n - 1}\delta^{N\zeta} \quad \text{and} \quad \nu(C_{n}) \geq 2^{-n - 1}\delta^{N\zeta}, \quad 1 \leq n \leq N. \end{equation}
In particular, the first part of \eqref{b3} with $n = N$ shows that $(c_{1},\ldots,c_{N}) \in \Omega$, recall the notation \eqref{b7}. To begin with, we define
\begin{displaymath} C_{1} := \{c \in C : \nu^{N - 1}(\Omega_{c}) \geq 2^{-2} \delta^{N\zeta}\}, \end{displaymath}
and we choose an arbitrary element $c_{1} \in C_{1}$. Since
\begin{displaymath} \int \nu^{N - 1}(\Omega_{c}) \, d\nu(c) = \nu^{N}(\Omega) \geq 2^{-1}\delta^{N\zeta} \end{displaymath}
by \eqref{form12}, and the case $n = 0$ of \eqref{b6}, we observe that $\nu(C_{1}) \geq 2^{-2}\delta^{N\zeta}$ by Chebyshev's inequality. In particular $C_{1} \neq \emptyset$. We then define
\begin{displaymath} H_{1} := c_{1}B_{c_{1}}. \end{displaymath}
%To be clear about what "$B_{c_{1}}$" means, note that if $c_{1} \in C_{1}$, then $\Omega_{c} \neq \emptyset$, and hence there are (many) choices of $(c_{2},\ldots,c_{N}) \in C^{N - 1}$ such that $(c_{1},\ldots,c_{N}) \in \Omega$. Further, the definition of $B_{c_{1}}$ in \eqref{form14} does not depend on the particular choice of $(c_{2},\cdots,c_{N}) \in C^{N - 1}$ such that $(c_{1},\ldots,c_{N}) \in \Omega$.

Assume inductively that $H_{1},\ldots,H_{n}$ and $C_{1},\ldots,C_{n} \subset C$, and $c_{j} \in C_{j}$, $1 \leq j \leq n \leq N - 1$, have already been constructed, and satisfy \eqref{b3}. We then pick an element $c_{n + 1} \in C_{n + 1}$, where
\begin{displaymath} C_{n + 1} := \{c \in C : \nu^{N - n - 1}(\Omega_{c_{1}\cdots c_{n} c}) \geq 2^{-n - 2}\delta^{N\zeta}\}, \quad 1 \leq n \leq N - 1. \end{displaymath}
For $n = N - 1$, the notation $\nu^{N - n - 1}(\Omega_{c_{1}\cdots c_{n}c})$ should be interpreted as in \eqref{b7}, so 
\begin{displaymath} C_{N} = \{c \in C : \mathbf{1}_{\Omega}(c_{1},\ldots,c_{N - 1},c) \geq 2^{-N - 1}\delta^{N\zeta}\} = \{c \in C : (c_{1},\ldots,c_{N - 1},c) \in \Omega\}. \end{displaymath}
For an arbitrary choice $c_{n + 1} \in C_{n + 1}$, we note that the first part of \eqref{b3} is satisfied with index "$n + 1$", simply by the definition of $C_{n + 1}$. 

The set $C_{n + 1}$ also satisfies the second part of \eqref{b3} with index "$n + 1$", by
\begin{displaymath} 2^{-n - 1}\delta^{N\zeta} \stackrel{\eqref{b3}}{\leq} \nu^{N - n}(\Omega_{c_{1}\cdots c_{n}}) \stackrel{\eqref{b6}}{=} \int \nu^{N - n - 1}(\Omega_{c_{1}\cdots c_{n}c}) \, d\nu(c), \end{displaymath}
and Chebyshev's inequality.

Whereas $c_{1} \in C_{1}$ was chosen arbitrarily, the element $c_{n + 1} \in C_{n + 1}$ is chosen in such a way that the quantity $|H_{n} + c_{n + 1}B_{c_{1}\cdots c_{n + 1}}|_{\delta}$ is maximised, among all possible choices $c_{n + 1} \in C_{n + 1}$. For this choice of $c_{n + 1} \in C_{n + 1}$, we define
\begin{displaymath} H_{n + 1} := (H_{n} + c_{n + 1}B_{c_{1}\cdots c_{n + 1}})_{\delta}. \end{displaymath}
Proceeding in this manner yields a sequence of sets $H_{1},\ldots,H_{N}$, and a distinguished sequence $(c_{1},\ldots,c_{N}) \in \Omega$, which we fix for the remainder of the argument. We record that if $(c_{1},\cdots,c_{n})$, $1 \leq n \leq N - 1$, is an initial sequence of $(c_{1},\cdots,c_{N})$, then
\begin{equation}\label{form13} |B_{c_{1}\cdots c_{n}c}| \geq \tfrac{1}{2}\delta^{N\zeta} |B_{c_{1}\cdots c_{n}}| \geq \delta^{\bar{\epsilon}}|B_{c_{1}\cdots c_{n}}|, \qquad c \in C_{n + 1}. \end{equation}
The second inequality simply follows from our choice of $\zeta$ at \eqref{b9}. To see the first inequality, recall from the definition of $c \in C_{n + 1}$ that (in particular) $\Omega_{c_{1}\cdots c_{n}c} \neq \emptyset$ (in the case $n = N - 1$ simply $(c_{1},\ldots,c_{n},c) \in \Omega$). This means that there exists a sequence $(c_{n + 2}',\ldots,c_{N}') \in C^{N - n - 1}$ such that $(c_{1},\ldots c_{n},c,c_{n + 2}',\ldots,c_{N}') \in \Omega$. Consequently,
\begin{displaymath} |B_{c_{1}\cdots c_{n}c}| \geq |B_{c_{1}} \cap \cdots B_{c_{n}} \cap B_{c} \cap B_{c_{n + 2}'} \cap \cdots B_{c_{N}'}| \geq \tfrac{1}{2}\delta^{N\zeta}|B| \geq \tfrac{1}{2}\delta^{N\zeta}|B_{c_{1}\cdots c_{n}}| \end{displaymath}
by the definition of $\Omega$, see \eqref{def:omega}.

Note that $H_{n} \subset (\delta \cdot \Z) \cap [0,n]$ for all $1 \leq n \leq N$ by a straightforward induction, so $|H_{n}| \leq n \delta^{-1}$. Therefore, by the pigeonhole principle, there exists an $n \in \{1,\ldots,N - 1\}$ such that
\begin{equation}\label{form22} |H_{n + 1}| \leq (N\delta^{-1})^{1/(N - 1)}|H_{n}| \leq 2\delta^{-1/(N - 1)}|H_{n}|. \end{equation}
We now consider the objects 
\begin{equation}\label{form9} \bar{A} := H_{n}, \quad \bar{B} := B_{c_{1}\cdots c_{n}}, \quad \text{and} \quad \bar{\nu} := \nu(C_{n + 1})^{-1}\nu|_{C_{n + 1}}. \end{equation}
We will show in a moment these objects satisfy the hypotheses of Theorem \ref{mainSubset1} with constants $\bar{\alpha},\bar{\beta},\kappa/2,\bar{\gamma}$, and $\bar{\epsilon}_{0}$. First, however, we conclude the proof of Theorem \ref{main}, taking this for granted. By Theorem \ref{mainSubset1}, there exists $\bar{c} \in C_{n + 1}$ (a set of full $\bar{\nu}$ measure) such that whenever $B' \subset \bar{B}$ is a set of cardinality $|B'| \geq \delta^{\bar{\epsilon}}|B|$, we have 
\begin{equation}\label{form15} |H_{n} + \bar{c}B'|_{\delta} = |\bar{A} + \bar{c}B'|_{\delta} \geq \delta^{-\bar{\epsilon}}|\bar{A}| = \delta^{-\bar{\epsilon}}|H_{n}|. \end{equation}
(To be accurate, Theorem \ref{mainSubset1} only claims this for some $\bar{c} \in \spt(\bar{\nu})$, but the proof showed, see \eqref{form25}, that actually the set of non-admissible $c \in \spt(\bar{\nu})$ have measure strictly smaller than $1$, so we can pick $c \in C_{n + 1}$.) However, for every $c \in C_{n + 1}$, the set $B' := B_{c_{1}\cdots c_{n}c} \subset B_{c_{1}\cdots c_{n}} = \bar{B}$ satisfies
\begin{equation}\label{form16} |B'| \stackrel{\eqref{form13}}{\geq} \delta^{\bar{\epsilon}}|\bar{B}| \quad \text{and} \quad |H_{n} + cB'|_{\delta} \lesssim |H_{n + 1}| \stackrel{\eqref{form22}}{\leq} 2\delta^{-1/(N - 1)}|H_{n}| \stackrel{\eqref{b23}}{\leq} \delta^{-\bar{\epsilon}/2}|H_{n}|. \end{equation}
The inequality $|H_{n} + cB'| \lesssim |H_{n + 1}|$ follows from the fact that whenever $c \in C_{n + 1}$, the set $H_{n} + c B' = H_{n} + c B_{c_{1}\cdots c_{n}c}$ is a competitor in the definition of $H_{n + 1}$. With the choice $c = \bar{c} \in C_{n + 1}$, the inequalities \eqref{form15}-\eqref{form16} are mutually incompatible for $\delta > 0$ small enough, depending on $\bar{\epsilon} = \bar{\epsilon}(\alpha,\beta,\gamma,\kappa) > 0$. A contradiction has been reached.

It remains to check that that the objects in \eqref{form22} satisfy the hypotheses of Theorem \ref{mainSubset1} with constants $\bar{\alpha},\bar{\beta},\kappa/2,\bar{\gamma}$, and $\bar{\epsilon}_{0}$. More precisely:
\begin{itemize}
\item[(a)] $|\bar{A}| \leq \delta^{-\bar{\alpha}}$,
\item[(b)] $|\bar{B}| \geq \delta^{-\bar{\beta}}$, and $\bar{B}$ satisfies a Frostman condition with exponent $\kappa/2$, for $r \in [\delta,\delta^{\bar{\epsilon}_{0}}]$,
\item[(c)] $\bar{\nu}$ satisfies a Frostman condition with exponent $\bar{\gamma}$.
\end{itemize}
We first use the Pl\"unnecke-Ruzsa inequality to establish (a), assuming that $\delta > 0$ is sufficiently small in terms of $N,\bar{\alpha}$. It is clear by induction that $H_{n}$ can be written as a sum of $n \leq N$ sets of the form $c_{m}B_{c_{1}\cdots c_{m}}$, for some $1 \leq m \leq n$. Noting that $A_{c_{1}\cdots c_{n}} \subset A_{c_{m}} \subset A$, each of these sets individually satisfies
\begin{displaymath} |A_{c_{1}\cdots c_{n}} + c_{m}B_{c_{1}\cdots c_{m}}|_{\delta} \leq |A_{c_{m}} + c_{m}B_{c_{m}}|_{\delta} \stackrel{\eqref{b5}}{\leq} \delta^{-\zeta}|A| \stackrel{\eqref{b4}}{\leq} 2\delta^{-(N + 1)\zeta}|A_{c_{1}\cdots c_{n}}|. \end{displaymath}
We may therefore infer that
\begin{displaymath} |H_{n}| \lesssim_{N} \delta^{-N(N + 1)\zeta}|A| \leq \delta^{-N(N + 1)\zeta - \alpha}. \end{displaymath}
from the Pl\"unnecke-Ruzsa inequality, Lemma \ref{PRIneq}, applied with $A_{c_{1}\cdots c_{n}}$ in place of $A$ (and finally also using $|A_{c_{1}\cdots c_{n}}| \leq |A| \leq \delta^{-\alpha}$, see above \eqref{form10}). This inequality implies $|H_{n}| \leq \delta^{-\bar{\alpha}}$ for small enough $\delta > 0$, recalling our choice of $\zeta$ at \eqref{b9}.

We move to (b). Recall from \eqref{form10} that the set $B$ satisfies the assumptions of Theorem \ref{main} with constants $\epsilon_{0},\kappa > 0$:
\begin{displaymath} |B \cap B(x,r)| \leq r^{\kappa}|B|, \qquad x \in \R, \, \delta \leq r \leq \delta^{\epsilon_{0}} = \delta^{\bar{\epsilon}_{0}}. \end{displaymath}
Since $B_{c_{1}\cdots c_{n}} \subset B$, and $|B_{c_{1}\cdots c_{n}}| \geq \tfrac{1}{2}\delta^{N\zeta}|B|$ by \eqref{b4}, we deduce that $B_{c_{1}\cdots c_{n}}$ satisfies a Frostman condition with exponent $\kappa/2$:
\begin{displaymath} |B_{c_{1}\cdots c_{n}} \cap B(x,r)| \leq 2\delta^{-N\zeta}r^{\kappa}|B_{c_{1}\cdots c_{n}}| \leq r^{\kappa/2}|B_{c_{1}\cdots c_{n}}|, \qquad x \in \R, \, \delta \leq r \leq \delta^{\bar{\epsilon}_{0}}. \end{displaymath}
The final inequality uses our choice of $\zeta$ in \eqref{b9}, and also assumes that $\delta > 0$ is sufficiently small, depending on $\bar{\epsilon}_{0},\kappa$. Moreover, since $|B| \geq \delta^{-\beta}$ by assumption, we have
\begin{displaymath} |B_{c_{1}\cdots c_{n}}| \geq \tfrac{1}{2}\delta^{N\zeta}|B| \stackrel{\eqref{b9}}{\geq} \delta^{-\bar{\beta}}. \end{displaymath}
Let us finally check (c), namely that the probability measure $\bar{\nu} = \nu(C_{n + 1})^{-1}\nu|_{C_{n + 1}}$ satisfies a Frostman condition with exponent $\bar{\gamma}$. Indeed, recalling from \eqref{b3} that $\nu(C_{n + 1}) \geq 2^{-n - 2}\delta^{N\zeta}$, we have
\begin{displaymath} \bar{\nu}(B(x,r)) \leq 2^{n + 2}\delta^{-N\zeta}\nu(B(x,r)) \leq 2^{N + 2}\delta^{-N\zeta} \cdot r^{\gamma}, \qquad x \in \R, \, \delta \leq r \leq \delta^{\bar{\epsilon}_{0}}. \end{displaymath}
Since $r^{\gamma} \leq \delta^{\epsilon_{0}(\gamma - \bar{\gamma})}r^{\bar{\gamma}}$ for $r \leq \delta^{\epsilon_{0}}$, by our choice of $\zeta$ in \eqref{b9}, the right hand side is bounded from above by $r^{\bar{\gamma}}$ for all $\delta > 0$ small enough, depending on $N,\gamma,\bar{\gamma}$ (all of which only depend on $\alpha,\beta,\gamma,\kappa$). We have now verified that the objects $\bar{A},\bar{B},\bar{\nu}$ from \eqref{form9} indeed satisfy the hypotheses of Theorem \ref{mainSubset1}. This concludes the proof of Theorem \ref{main}. \end{proof}

\section{Proof of Corollary \ref{hausdorffCor}}\label{appA}

Here is the statement once more:
\begin{cor} Let $0 < \beta \leq \alpha < 1$ and $\kappa > 0$. Then, there exists $\eta = \eta(\alpha,\beta,\kappa) > 0$ such that if $A,B \subset \R$ are Borel sets with $\Hd A = \alpha$, $\Hd B = \beta$, then 
\begin{displaymath} \Hd \{c \in \R : \Hd (A + cB) \leq \alpha + \eta\} \leq \tfrac{\alpha - \beta}{1 - \beta} + \kappa. \end{displaymath}
\end{cor}

\begin{proof}[Proof of Corollary \ref{hausdorffCor} assuming Theorem \ref{main}] It is easy to reduce to the case where $A,B$ are compact, $A,B \subset [0,1]$, and $\mathcal{H}^{\alpha}(A) > 0$ and $\mathcal{H}^{\beta}(B) > 0$. In this case, one may use Frostman's lemma \cite[Theorem 8.8]{zbMATH01249699} to find Borel probability measures $\mu_{A},\mu_{B}$ with $\spt(\mu_{A}) \subset A$, $\spt (\mu_{B}) \subset B$, and satisfying $\mu_{A}(B(x,r)) \leq C_{A}r^{\alpha}$ and $\mu_{B}(B(x,r)) \leq C_{B}r^{\beta}$ for all balls $B(x,r) \subset \R$. If $\eta > 0$ is small enough, we will show that $\Hd E \leq (\alpha - \beta)/(1 - \beta) + \kappa$, where
\begin{displaymath} E := E_{\eta} := \{c \in [\tfrac{1}{2},1] : \Hd (A + cB) < \alpha + \eta\}. \end{displaymath} 
It is easy to show (by rescaling considerations) that this implies Corollary \ref{hausdorffCor}, where $[\tfrac{1}{2},1]$ is replaced by $\R$. It is well-known that the set $E \subset [\tfrac{1}{2},1]$ is Borel. Consequently, if the inequality fails, one may use Frostman's lemma again to find a Borel probability measure $\nu$, supported on $E$, satisfying $\nu(B(x,r)) \leq C_{\nu}r^{\gamma}$ for all $x \in \R$ and $r > 0$, where $\gamma \geq (\alpha - \beta)/(1 - \beta) + \kappa$.

For future reference, we fix some parameters $\bar{\alpha} > \alpha$, $\bar{\beta} < \beta$, and $\bar{\gamma} < \gamma$ such that the inequality
\begin{equation}\label{form111} \bar{\gamma} > (\bar{\alpha} - \bar{\beta})/(1 - \bar{\beta}) \end{equation}
still holds. We then let $\bar{\epsilon},\bar{\epsilon}_{0},\bar{\delta}_{0} > 0$ be the constants provided by Theorem \ref{main} applied with parameters $\bar{\alpha},\bar{\beta},\kappa = \bar{\beta},\bar{\gamma}$. We pick $\eta > 0$ in the definition of $E$ so small that
\begin{equation}\label{form113} \eta < \min\{\bar{\epsilon},\bar{\alpha} - \alpha\}. \end{equation}

Fix $c \in \spt(\nu) \subset E$, so $\Hd (A + cB) < \alpha + \eta$. This means that for a given fixed threshold $\delta_{0} := 2^{-j_{0}} \in 2^{-\N}$ (the requirements will depend on $\alpha,\beta,\gamma,C_{A},C_{B},C_{\nu}$), one may find a countable cover $\mathcal{I}_{c}$ of $A + cB$, consisting of disjoint dyadic intervals of length $\ell(I) \leq \delta_{0}$, such that
\begin{equation}\label{form100} \sum_{I \in \mathcal{I}_{c}} \ell(I)^{\alpha + \eta} \leq 1. \end{equation} 
Below, we will often write that something holds "for small enough $\delta > 0$": this will always mean "assuming that the upper bound $\delta_{0}$ for $\delta$ has been chosen sufficiently small, depending on the parameters $\alpha,\beta,\gamma,C_{A},C_{B},C_{\nu}$. In particular, we will take $\delta_{0} \leq \bar{\delta}_{0}$.

The "tubes" $\mathcal{T}_{c} := \{\pi_{c}^{-1}(I)\}_{I \in \mathcal{I}_{c}}$ cover $A \times B \supset \spt(\mu_{A} \times \mu_{B})$, so
\begin{displaymath} \int_{E} \sum_{T \in \mathcal{T}_{c}} (\mu_{A} \times \mu_{B})(T) \, d\nu(c) = 1. \end{displaymath}
Recall that $\delta_{0} = 2^{-j_{0}}$, and let $\mathcal{I}_{c}^{j} := \{I \in \mathcal{I}_{c} : \ell(I) = 2^{-j}\}$ for $j \geq j_{0}$. Write also $\mathcal{T}^{j}_{c} := \{\pi_{c}^{-1}(I)\}_{I \in \mathcal{I}_{c}^{j}}$. Since $\mathcal{T}_{c} = \bigcup_{j \geq j_{0}} \mathcal{T}_{c}^{j}$, there exists $j \geq j_{0}$ such that
\begin{displaymath} \int_{E} \sum_{T \in \mathcal{T}_{c}^{j}} (\mu_{A} \times \mu_{B})(T) \, d\nu(c) \gtrsim j^{-2}. \end{displaymath}
Write $\delta := 2^{-j}$ for this index $j$. According to the estimate above, there exists a subset $E_{\delta}' \subset E$ of measure $\nu(E_{\delta}') \gtrsim j^{-2} = \log_{2}(1/\delta)^{-2}$ such that for each $c \in E_{\delta}'$, the tubes $T \in \mathcal{T}_{c}^{j}$ cover a subset $G_{c} \subset \spt (\mu_{A} \times \mu_{B})$ of measure $(\mu_{A} \times \mu_{B})(G_{c}) \gtrsim \log_{2}(1/\delta)^{-2}$. In particular, we record that
\begin{equation}\label{form99} |\pi_{c}(G_{c})|_{\delta} \leq |\mathcal{T}_{c}^{j}| \leq \delta^{-\alpha - \eta}, \qquad c \in E_{\delta}', \end{equation}
by \eqref{form100}. For the remainder of this argument, we use the notation $f \lessapprox g$ to abbreviate an inequality of the form $f \leq C\log_{2}(1/\delta)^{C}g$ for some constant $C > 0$, which may depend on the Frostman constants $\alpha,\beta,\gamma,C_{A},C_{B},C_{\nu}$. In particular, $j^{-2} = \log_{2}(1/\delta)^{-2} \gtrapprox 1$.

For $x \in \R$, let $I_{\delta}(x) \in \mathcal{D}_{\delta}$ be the unique dyadic interval of length $\delta$ with $x \in I_{\delta}(x)$. We now split the set $A$ as follows:
\begin{displaymath} A = \bigcup_{\rho \in 2^{-\N}} A(\rho) := \{x \in A : \rho \leq \mu_{A}(I_{\delta}(x)) < 2\rho\}. \end{displaymath}
We define the sets $B(\rho) \subset B$ similarly. Since $\mu_{A}(I_{\delta}(x)) \leq C_{A}\delta^{\alpha}$ and $\mu_{B}(I_{\delta}(y)) \leq C_{B}\delta^{\beta}$, we see that $A(\rho) \neq \emptyset$ implies $\rho \leq C_{A}\delta^{\alpha}$, and $B(\rho) \neq \emptyset$ implies $\rho \leq C_{\beta}\delta^{\beta}$. We also note that $A(\rho)$ can be expressed as the intersection of $A$ with certain dyadic intervals $\mathcal{A}(\rho) \subset \mathcal{D}_{\delta}$. The same is true for $B(\rho)$, for certain dyadic intervals $\mathcal{B}(\rho) \subset \mathcal{D}_{\delta}$.

Let $\mu_{A}(\rho)$ be the restriction of $\mu_{A}$ to the intervals $\mathcal{A}(\rho)$, and similarly let $\mu_{B}(\rho)$ be the restriction of $\mu_{B}$ to the intervals in $\mathcal{B}(\rho)$. Then 
\begin{equation}\label{form109} \sum_{\rho_{1}} \sum_{\rho_{2}} \int_{E_{\delta}'} (\mu_{A}(\rho_{1}) \times \mu_{B}(\rho_{2}))(G_{c}) \approx 1, \end{equation}
so it follows from the pigeonhole principle that 
\begin{displaymath} \int_{E_{\delta}'} (\mu_{A}(\rho_{A}) \times \mu_{B}(\rho_{A}))(G_{c}) \approx 1 \end{displaymath}
for some fixed choices $\rho_{A} \leq C_{A}\delta^{\alpha}$ and $\rho_{B} \leq C_{B}\delta^{\beta}$ (noting that values $\rho_{1},\rho_{2} \leq \delta^{2}$ cannot contribute substantially to the sum in \eqref{form109}). In particular, there exists a further subset $E_{\delta} \subset E_{\delta}'$ with the property $(\mu_{A}(\rho_{A}) \times \mu_{B}(\rho_{B}))(G_{c}) \approx 1$ for all $c \in E_{\delta}$. We now abbreviate
\begin{displaymath} \bar{\mu}_{A} := \mu_{A}(\rho_{A}) \quad \text{and} \quad \bar{\mu}_{B} := \mu_{B}(\rho_{B}), \end{displaymath}
so $\|\bar{\mu}_{A}\| \approx 1 \approx \|\bar{\mu}_{B}\|$. The measure $\bar{\mu}_{A}$ is supported on the closure of the intervals in $\mathcal{A}(\rho_{A})$, and $\bar{\mu}_{B}$ is supported on the closure of the intervals in $\mathcal{B}(\rho_{B})$. Let 
\begin{displaymath} A_{\delta} := (\delta \cdot \Z) \cap \left(\cup \mathcal{A}(\rho_{A}) \right) \quad \text{and} \quad B_{\delta} := (\delta \cdot \Z) \cap \left(\cup \mathcal{B}(\rho_{B}) \right). \end{displaymath}
We observe that 
\begin{equation}\label{form110} \rho_{A} \cdot |A_{\delta}| \sim \|\mu_{A}\| \approx 1 \quad \Longrightarrow \quad \rho_{A} \approx |A_{\delta}|^{-1}, \end{equation}
and similarly $\rho_{B} \approx |B_{\delta}|^{-1}$. Since $\rho_{A} \leq C_{A}\delta^{\alpha}$, we record that
\begin{equation}\label{form112} |A_{\delta}| \approx \rho_{A}^{-1} \gtrapprox \delta^{-\alpha}. \end{equation}
We next claim that, somewhat conversely, $|A_{\delta}| \leq \delta^{-\bar{\alpha}}$ if $\delta > 0$ is sufficiently small. To see this, fix an arbitrary $c \in E_{\delta}$. Since $(\bar{\mu}_{A} \times \bar{\mu}_{B})(G_{c}) \approx 1$, there exists $b \in \spt (\bar{\mu}_{B})$ such that
\begin{displaymath} \bar{\mu}_{A}(G_{c}(b)) \approx 1, \quad \text{where} \quad  G_{c}(b) = \{x \in \spt(\bar{\mu}_{A}) : (x,b) \in G_{c}\}. \end{displaymath}
Now, if $\mathcal{G}_{c}(b) := \{I \in \mathcal{A}(\rho_{A}) : G_{c}(b) \cap I \neq \emptyset\}$, we see that $\bar{\mu}_{A}(I) \sim \rho_{A}$ for all $I \in \mathcal{G}_{c}(b)$, and $\bar{\mu}_{A}(\cup \mathcal{G}_{c}(b)) \geq \bar{\mu}_{A}(G_{c}(b)) \approx 1$. Moreover, we observe that $|G_{c}(b)|_{\delta} \lesssim |\pi_{c}(G_{c})|_{\delta}$, since $\pi_{c}(G_{c}) \supset G_{c}(b) + bc$. Putting these observations together,
\begin{equation}\label{form101} |A_{\delta}| \stackrel{\eqref{form110}}{\approx} \rho_{A}^{-1} \lessapprox \rho_{A}^{-1} \cdot \bar{\mu}_{A}(\cup \mathcal{G}_{c}(b)) \lesssim |G_{c}(b)|_{\delta} \lesssim |\pi_{c}(G_{c})|_{\delta} \stackrel{\eqref{form99}}{\leq} \delta^{-\alpha - \eta}. \end{equation}
Since $\alpha + \eta < \bar{\alpha}$ by \eqref{form113}, the inequality $|A_{\delta}| \leq \delta^{-\bar{\alpha}}$ holds for $\delta > 0$ sufficiently small.

Next, since $\rho_{B} \leq C_{B}\delta^{\beta}$, we record that 
\begin{equation}\label{form102} |B_{\delta}| \approx \rho_{B}^{-1} \gtrapprox \delta^{-\beta} \quad \Longrightarrow \quad |B_{\delta}| \geq \delta^{\bar{\beta}}, \end{equation}
where the implication holds if $\delta > 0$ is sufficiently small. Moreover, for $x \in \R$ and $r \geq \delta$, we note that every point $y \in B_{\delta} \cap B(x,r)$ is contained in an interval $I_{y}(\delta) \in \mathcal{B}(\rho_{B})$ with $\mu_{B}(I_{y}(\delta)) \geq \rho_{B}$. Since $I_{y}(\delta) \subset B(x,2r)$, we deduce that 
\begin{equation}\label{form103} |B_{\delta} \cap B(x,r)| \leq \rho_{B}^{-1} \cdot \mu_{B}(B(x,2r)) \leq \rho_{B}^{-1} \cdot C_{B}(2r)^{\beta} \lessapprox r^{\beta}|B_{\delta}|. \end{equation}
In particular, for the parameter $\bar{\epsilon}_{0} > 0$ fixed below \eqref{form111}, we have $|B_{\delta} \cap B(x,r)| \leq r^{\bar{\beta}}|B_{\delta}|$ for $\delta \leq r \leq \delta^{\bar{\epsilon}_{0}}$, provided that $\delta > 0$ is small enough.

Finally, the measure $\nu_{\delta} := \nu(E_{\delta})^{-1} \cdot \nu|_{E_{\delta}}$ satisfies
\begin{equation}\label{form104} \nu_{\delta}(B(x,r)) \lessapprox \nu(B(x,r)) \leq C_{\nu}r^{\gamma}, \qquad r > 0, \end{equation}
so the inequality $\nu_{\delta}(B(x,r)) \leq r^{-\bar{\gamma}}$ holds for all $r \leq \delta^{\bar{\epsilon}_{0}}$, provided that $\delta > 0$ is small enough. The estimates \eqref{form101}-\eqref{form104}, and \eqref{form111}, imply that the triple $A_{\delta},B_{\delta},\nu_{\delta}$ satisfies all the hypotheses of Theorem \ref{main} with constants $\bar{\alpha},\bar{\beta},\kappa = \bar{\beta},\bar{\gamma}$, and $\bar{\epsilon}_{0}$. Consequently, there exists $c \in E_{\delta} \subset E_{\delta}'$ (a set of full $\nu_{\delta}$ measure) such that
\begin{equation}\label{form114} |\pi_{c}(G)|_{\delta} \geq \delta^{-\bar{\epsilon}}|A_{\delta}| \stackrel{\eqref{form112}}{\gtrapprox} \delta^{-\alpha - \bar{\epsilon}} \end{equation}
for all subsets $G \subset A_{\delta} \times B_{\delta}$ of cardinality $|G| \geq \delta^{\bar{\epsilon}}|A||B|$. We argue that this contradicts \eqref{form99}. The only issue is that set $G_{c} \subset \spt(\mu_{A} \times \mu_{B})$ is not exactly a subset of $A_{\delta} \times B_{\delta}$. To fix this, recall that nevertheless $(\bar{\mu}_{A} \times \bar{\mu}_{B})(G_{c}) \approx 1$. Let 
\begin{displaymath} \mathcal{G}_{c} := \{I \times J \in \mathcal{A}(\rho_{A}) \times \mathcal{B}(\rho_{B}) : (I \times J) \cap G_{c} \neq \emptyset\}. \end{displaymath}
Then $\mathcal{G}_{c}$ is a cover of $G_{c}$, and $(\bar{\mu}_{A} \times \bar{\mu}_{B})(Q) \sim \rho_{A} \rho_{B} \approx |A_{\delta}|^{-1}|B_{\delta}|^{-1}$ for all $Q = I \times J \in \mathcal{G}_{c}$. Consequently,
\begin{displaymath} |\mathcal{G}_{c}| \gtrsim (\rho_{A}\rho_{B})^{-1} \cdot (\bar{\mu}_{A} \times \bar{\mu}_{B})(G_{c}) \approx |A_{\delta}||B_{\delta}|. \end{displaymath}
Now, let $G_{c,\delta} \subset (A_{\delta} \times B_{\delta}) \cap G_{c}(2\delta)$ be subset of cardinality $|G_{c,\delta}| \gtrapprox |A_{\delta}||B_{\delta}|$. In particular $|G_{c,\delta}| \geq \delta^{\bar{\epsilon}}|A_{\delta}||B_{\delta}|$ for $\delta > 0$ small enough. Therefore the estimate \eqref{form114} holds for $G = G_{c,\delta}$. On the other hand, since $G_{c,\delta} \subset G_{c}(2\delta)$, we have
\begin{displaymath} |\pi_{c}(G_{c,\delta})|_{\delta} \lesssim |\pi_{c}(G_{c})|_{\delta} \leq \delta^{-\alpha - \eta} \end{displaymath}
by \eqref{form99}. Since we chose $\eta < \bar{\epsilon}$ in \eqref{form113}, this estimate is not compatible with \eqref{form114}. A contradiction has been reached, and the proof of Corollary \ref{hausdorffCor} is complete. \end{proof}

\bibliographystyle{plain}
\bibliography{references}

\end{document}